\documentclass[11pt]{article}
\usepackage{amsthm}
\usepackage{amsfonts}
\usepackage[T1]{fontenc}
\usepackage{epsfig}

\usepackage[lflt]{floatflt}
\usepackage{epsfig}
\usepackage{graphicx}
\usepackage{amsmath}
\usepackage{amssymb}
\usepackage{color}
\usepackage{subfigure}
\usepackage[english]{babel}
\usepackage{empheq}
\usepackage{caption}

\usepackage{hyperref}

\usepackage{makeidx}
\usepackage{latexsym}
\usepackage{enumerate}
\newtheorem{theo}{Theorem}
\newtheorem{defi}{Definition}
\newtheorem{prop}{Proposition}
\newtheorem{lemma}{Lemma}
\newtheorem{coro}{Corollary}
\newtheorem{remark}{Remark}

\newcommand{\bbN}{{\mathbb N}}
\newcommand{\bbR}{{\mathbb R}}
\newcommand{\bbC}{{\mathbb C}}

\newcommand{\al}{\alpha}

\newcommand{\p}{\partial}

\newcommand{\dd}{{\rm d}}
\newcommand{\ia}{I^{\alpha}}
\newcommand{\izj}{\int_{0}^{1}}

\newcommand{\da}{D^{\alpha}}
\newcommand*{\poch}{\frac{\p}{\p x}}
\newcommand{\ddd}{\widetilde{\mathcal{D}}_{\al}}
\newcommand*{\norm}[1]{\left\Vert{#1}\right\Vert}
\def\arg{\operatorname{arg}}
\newcommand*{\abs}[1]{\left\vert{#1}\right\vert}
\newcommand*{\uz}{u_{0}}
\newcommand{\eqq}[2]{\begin{equation}  #1  \label{#2} \end{equation}}
\newcommand{\hd}{\hspace{0.2cm}}
\newcommand{\m}[1]{\mbox{#1}}
\newcommand{\izx}{\int_{0}^{x}}
\newcommand{\ld}{L^{2}(0,1)}

\newcommand{\vf}{\varphi}
\newcommand{\ve}{\varepsilon}

\newcommand{\ca}{c_{\al}}
\newcommand{\ba}{b_{\al}}

\newcommand\restr[2]{{
  \left.\kern-\nulldelimiterspace 
  #1 
  \littletaller 
  \right|_{#2} 
  }}

\newcommand{\littletaller}{\mathchoice{\vphantom{\big|}}{}{}{}}

\usepackage{anysize}
\marginsize{1.5cm}{1.5cm}{1cm}{1cm}

\begin{document}

\begin{center}\emph{}
\LARGE
\textbf{On a Space Fractional Stefan problem of Dirichlet type with Caputo flux}
\end{center}
\normalsize
\begin{center}
S. D. Roscani$^{2,3}$, K. Ryszewska$^1$ and L.D. Venturato$^{2,3}$\\
$^1$ Department of Mathematics and Information Sciences, Warsaw University of Technology, Koszykowa 75, 00-662 Warsaw, Poland\\
$^2$ Departamento de Matem\'atica, FCE, Universidad Austral, Paraguay 1950, S2000FZF Rosario, Argentina \\
$^3$CONICET, Argentina\\
(sroscani@austral.edu.ar, Katarzyna.Ryszewska@pw.edu.pl, lventurato@austral.edu.ar)
\vspace{0.2cm}
\end{center}

\small

\textbf{Abstract:}
We study a space-fractional Stefan problem with the Dirichlet boundary conditions. It is a model that describes superdiffusive phenomena.  
Our main result is the existence of the unique classical solution to this problem. In the proof we apply evolution operators theory and the Schauder fixed point theorem. 
It appears that studying fractional Stefan problem with Dirichlet boundary conditions requires a substantial modifications of the approach in comparison with the existing results for problems with different kinds of boundary conditions. \\



\normalsize

\section{Introduction}

In recent years, there has been a significant increase in research interests on free boundary problems governed by fractional diffusion equations. These problems can be classified as time-fractional or space-fractional, depending in which variable the fractional derivative is taken. 
Applied to the time variable, applications are related to memory effects, while in the second case ones speaks about nonlocality in space. These anomalous phenomena are often associated with heterogeneity in the material or media considered.

In this paper we focus on a problem governed by a space-fractional diffusion equation.
Several models related to these kind of problems have been extensively studied. Some examples include dendritic crystal growth \cite{JuTry, KaRa}, shoreline dynamics in ocean basins \cite{CaBeYo, MaSwePaVo}, scaled Brownian motion \cite{GeVoMiDa:2013}, anomalous transport processes \cite{MetKla}, and infiltration of water into heterogeneous sub-surface soils \cite{Voller2011}.
Notice that by considering different diffusive fluxes one can obtain distinct governing equations, as proposed in \cite{NaRyVo2021}, where a non-local transport phenomenon is analyzed for two non-local fluxes defined in terms of fractional derivatives.

An important difference between classical and fractional one-phase Stefan problems in one-dimension arises in the behavior of the free boundary. In classical Stefan problems with constant Dirichlet boundary condition on fixed boundary, the interface grows as $s(t)\sim t^{\frac{1}{2}}$. However, for certain fractional cases, it has been observed that the phase change boundary moves as $s(t)\sim t^{\beta}$ where $0<\beta<\frac{1}{2}$ (sub-diffusion) or $\frac{1}{2}<\beta<1$ (super-diffusion). This behavior appears, for example, when heterogeneities occur across a range of length scales, with the largest approaching the domain scale, causing standard diffusion models to break down \cite{MetKla, SchuMeBa}.
Concerning theoretical results, the self-similar solutions for a fractional Stefan problem with fractional Neumann-type condition were recently obtained in \cite{RoTaVe2020}, where the solution verifies that the advance of the free boundary is proportional to $t^\frac{1}{1+\al}$, $\al\in(0,1)$ (super-diffusion), which coincides with the result obtained in \cite{Voller2014} for the stationary case.

Here, we analyze a space-fractional Stefan problem with a Dirichlet boundary conditions. More precisely, we study existence and uniqueness of classical solutions to the following fractional Stefan problem
\begin{equation}\label{StefanProblem0}
\begin{array}{lll}
(i) & u_t-\frac{\partial}{\partial x}D^\al u=0, & 0<x<s(t), 0<t<T,\\
(ii) & u(0,t)=0, u(t,s(t))=0, &   t\in (0,T),\\
(iii) & u(x,0)=u_0(x), &  0<x<s(0)=b,\\
(iv) & \dot{s}(t)=-\lim\limits_{x\rightarrow s(t)^-}(D^\al u)(x,t), & t\in (0,T).
\end{array}
\end{equation}
for $b>0$. 

A formal derivation of \eqref{StefanProblem0} can be obtained from physical assumptions on a model of a melting phenomena related to a phase change in an infinite slab, due to heat transfer (see \cite{RoTaVe2020}).

One has to mention that in \cite{RoRyVe2022} and  \cite{Rys:2020} the space-fractional Stefan problems with local and nonlocal Neumann boundary conditions are considered. Although here we only change a boundary condition, it appears that the techniques used in the former articles have to be significantly adapted and modified in this case. Moreover, it is well known that the different kinds of boundary conditions  must be treated very carefully when a space fractional derivative is involved (see for example \cite{BaeuLuksMeer} for fractional diffusion or  \cite[Ch.1]{Yama:2022} for general fractional  ordinary differential equations) and, according to our best knowledge, this is the first existence result for the space-fractional Stefan problem with the Caputo derivative and Dirichlet boundary conditions.

We note that a self-similar solution for problem \eqref{StefanProblem0} with $b=0$ and non-homogeneous boundary conditions, was obtained in \cite{RoTaVe2020}. Furthermore, the viscosity solution to a linear 
initial boundary value problem $u_t = \poch \da u$ with Dirichlet boundary conditions have been obtained in \cite{NaRy2020} by Perron's method, and the explicit solutions for fractional diffusion problems on bounded domains have been established in \cite{BaeuLuksMeer}. The analytic semigroup theory for linear evolution problem with $\poch \da$ with different kinds of boundary conditions was established in a PhD thesis of one of the authors (\cite{PhDKasia}). 
Finally, there are several results concerning alternative problems where the fractional Laplace operator is considered \cite{AcoBorBruMaas, BaFiRos, CaSoVa, TeEnVa2020}.

The goal of this paper is solve the problem \eqref{StefanProblem0}.
To this end, we adapt the results from  \cite{RoRyVe2022}, \cite{Rys:2020}, and establish analogous properties and estimates in our case. However, the case with Dirichlet boundary conditions requires substantial change of the approach in some crucial parts of the proof. For example, it is important to notice that discussing \eqref{StefanProblem0}, we are not able to control the fractional flux at the boundary $(\da u)(0,t)$. Thus, in order to solve the problem via fixed point argument, we are forced to deduce a new integral identity equivalent to the Stefan condition \eqref{StefanProblem0}($iv$), which does not include $(\da u)(0,t)$. Furthermore, since we work with a very specific functional spaces (see (\ref{domain})) and our techniques involve interpolation estimates, we have to provide a local estimate between our specific interpolation spaces and fractional Sobolev spaces (Theorem~ \ref{IneqInterpSp}). Our approach relies on finding mild solutions to the moving boundary problem, increasing their regularity and then applying the fixed point argument. Increasing the regularity of mild solutions is very demanding and technical task, which is illustrated in \cite{RoRyVe2022} and \cite{Rys:2020}. The situation is getting particularly complex when $0 < \al \leq 1/2$ and the iterative argument has to be used. Since the iterative procedure has already been applied twice in \cite[Lemma 3.10]{RoRyVe2022} and \cite[Lemma 5]{Rys:2020}, in this paper we discuss only the case $\al > 1/2$. It helps us focus on the novelties which problem (\ref{StefanProblem0}) delivers, without accumulating too much of technical calculations. Notice that the assumption $\al > 1/2$ is only used in the regularity results (Section \ref{subsecReg}) and the rest of the proof is valid for any $0<\al<1$.


The paper is organised as follows. In section 2, we recall some definitions and properties of fractional operators and fractional Sobolev spaces.
In section 3, we firstly establish the analytic semigroup theory for the operator $\poch \da$ with Dirichlet boundary conditions. Then we apply this result to solve the moving boundary problem asociated to \eqref{StefanProblem0}. 
Finally, section 4 is devoted to the proof of existence and uniqueness of the solutions to space-fractional Stefan problem (\ref{StefanProblem0}) for $\al > 1/2$. The proof follows the steps from \cite{RoRyVe2022}  and \cite{Rys:2020}, but here we need to apply different integral identity equivalent to the Stefan condition for problem \eqref{StefanProblem0}.


\section{Preliminaries}

In this section, we recall some definitions and properties concerning fractional derivatives and fractional Sobolev spaces. Let us begin with definitions of fractional operators. We notice that the lower extreme of the fractional operators considered in this article is always zero, thus the classical suffix 0 will be omitted.  

\begin{defi}
Let $\al>0$. For $f\in L^1(0,1)$, we define the fractional integral $I^\al$ as follows
\begin{equation}\label{defIntAlfa}
I^\al f(x)=\frac{1}{\Gamma(\al)}\int_0^x (x-p)^{\al-1}f(p)\dd p.
\end{equation}
\end{defi}

\begin{defi}\label{defDersFrac}
Let $\al\in(0,1)$. For $f$ regular enough, the fractional Riemann-Liouville is defined by the formula
\begin{equation}\label{defDerRL}
\partial^\al f(x)=\frac{\partial}{\partial x}I^{1-\al}f(x)=\frac{1}{\Gamma(1-\al)}\frac{\partial}{\partial x}\int_0^x (x-p)^{-\al}f(p)\dd p,
\end{equation}
while the fractional Caputo derivative is given by 
\begin{equation}\label{defDerC}
D^\al f(x)=\frac{\partial}{\partial x}I^{1-\al}[f(x)-f(0)]=\frac{1}{\Gamma(1-\al)}\frac{\partial}{\partial x}\int_0^x (x-p)^{-\al}[f(p)-f(0)]\dd p.
\end{equation}
\end{defi}

\begin{remark}\label{RelacIFyDF}
The Riemann-Liouville and Caputo fractional derivatives are well defined for functions belonging to $AC[0,1]$. Moreover, if $f\in AC[0,1]$ then
\begin{equation}\label{eqformDC}
D^\al f(x)=I^{1-\al}f'(x)=\frac{1}{\Gamma(1-\al)}\int_0^x (x-p)^{-\al}f'(p)\dd p.
\end{equation}
\end{remark}

The absolute continuity of a function is a sufficient condition to make $\partial^{\al}$ pointwisely well defined but it is not the necessary one. In the following proposition we give an example of singular function which has pointwise Riemann-Liouville fractional derivative and present some basic properties of fractional operators which will be used later on.

\begin{prop}\label{IFpowers}\cite[Chapter 2]{Samko}, \cite[Proposition 6.5]{KuYa}

\begin{enumerate}

\item Let $\beta> -1$. Then,
$$I^\al(x^\beta)=\frac{\Gamma(\beta+1)}{\Gamma(\al+\beta+1)}x^{\al+\beta},$$
and
$$\p^\al(x^\beta)=\frac{\Gamma(\beta+1)}{\Gamma(\beta-\al+1)}x^{\beta-\al},$$

\item  Let $f\in AC[0,1]$. Then,
\begin{equation}\label{darl}
(\da f)(x) = (\p^{\al}f)(x) - \frac{x^{-\al}}{\Gamma(1-\al)}f(0),
\end{equation}
and
\begin{equation}\label{equiv2}
I^{1-\al}f(x)=\int_0^x D^{\al} f(z)\dd z+f(0)\frac{x^{1-\al}}{\Gamma(2-\al)}.
\end{equation}
In particular, if $f(0)=0$,
\begin{equation}\label{equiv2-1}
I^{1-\al}f(x)=\int_0^x D^{\al} f(z)\dd z.
\end{equation}
Moreover,
\begin{equation}\label{zam}
\poch \da u = \poch I^{1-\al} u_{x} = \p^{\al}u_{x}.
\end{equation}
\item For every  $f \in AC[0,1]$, $\al +\beta \leq 1$ there holds \eqq{\p^{\alpha}D^{\beta} f=D^{\al+\beta}f, \quad a.e. \text{ in } (0,1), \textrm{ in particular } \p^{1-\alpha}D^{\al} f=f', \quad a.e. \text{ in } (0,1).}{super}
\end{enumerate}

\end{prop}

While working with fractional operators the fractional Sobolev spaces seems to be a natural choice.   Below we recall that the fractional Sobolev space $H^\beta(0,1):=W^{\beta,2}(0,1)$, $\beta\in (0,1)$ is defined as
$$H^\beta(0,1)=\left\{f\in L^2(0,1):\int_0^1\int_0^1\frac{|f(x)-f(y)|^2}{|x-y|^{1+2\beta}}\dd y\dd x <\infty\right\},$$
equipped with the norm
$$||f||_{H^\beta(0,1)}=\left(||f||^2_{L^2(0,1)}+\int_0^1\int_0^1\frac{|f(x)-f(y)|^2}{|x-y|^{1+2\beta}}\dd y\dd x\right)^{\frac{1}{2}}.$$

For $\beta>1$, $m=\lfloor \beta \rfloor$ and $s=\beta-\lfloor \beta \rfloor$, we define $H^\beta(0,1):=W^{\beta,2}(0,1)$ as follows
$$H^\beta(0,1)=\{f\in H^m(0,1) : f^{(m)} \in H^s(0,1)\},$$
with the norm
$$||f||_{H^\beta(0,1)}=\left(||f||^2_{H^m(0,1)}+||f^{(m)}||^2_{H^s(0,1)}\right)^{\frac{1}{2}}.$$

We will often make use of this basic property. 

\begin{prop}\cite[Remark 12.8]{Lions}\label{LMderiv}
For every $s\neq\frac{1}{2}$, there holds,
\begin{equation}\label{Lionsderiv0}
\partial_x\in \mathcal{L}(H^s(0,1),H^{s-1}(0,1)).
\end{equation}
\end{prop}

\begin{defi}\label{sobdef}
Let $\gamma>0$, and
$${_0}C^\infty[0,1]=\{v\in C^\infty[0,1]:\forall n\in \bbN, v^{(n)}(0)=0\},\quad {^0}C^\infty[0,1]=\{v\in C^\infty[0,1]:\forall n\in \bbN, v^{(n)}(1)=0\}.$$
We define the fractional Sobolev spaces ${_0}H^\gamma(0,1)$ and ${^0}H^\gamma(0,1)$ in the  following way
\begin{equation}
{_0}H^\gamma(0,1)=\overline{{_0}C^\infty[0,1]}^{H^\gamma(0,1)},\quad {^0}H^\gamma(0,1)=\overline{{^0}C^\infty[0,1]}^{H^\gamma(0,1)}.
\end{equation}
\end{defi}

Let us recall the characterization for fractional Sobolev spaces.

\begin{prop}\cite[Theorem 11.6 and Remark 11.5]{Lions}, see also \cite[Proposition 1]{Yama:2022}
For $\al \in (0,1)$ the spaces introduced in Definition \ref{sobdef} may be equivalently characterized as follows 
\begin{equation}
{_0}H^\al(0,1)=
\begin{cases}
H^\al(0,1), & \al \in \left(0,\frac{1}{2}\right),\\
\{u\in H^\frac{1}{2}(0,1): \int_0^1 \frac{|u(x)|^2}{x}\dd x<\infty \}, & \al =\frac{1}{2},\\
\{u\in H^\al(0,1): u(0)=0 \}, & \al \in \left(\frac{1}{2},1\right),
\end{cases}
\end{equation}
and
\begin{equation}
{^0}H^\al(0,1)=
\begin{cases}
H^\al(0,1), & \al \in \left(0,\frac{1}{2}\right),\\
\{u\in H^\frac{1}{2}(0,1): \int_0^1 \frac{|u(x)|^2}{1-x}\dd x<\infty \}, & \al =\frac{1}{2},\\
\{u\in H^\al(0,1): u(1)=0 \}, & \al \in \left(\frac{1}{2},1\right),
\end{cases}
\end{equation}
with
$$||u||_{{_0}H^\al(0,1)}=||u||_{{^0}H^\al(0,1)}=||u||_{H^\al(0,1)},\quad \al\neq \frac{1}{2},$$
$$||u||_{{_0}H^\frac{1}{2}(0,1)}=\left(||u||^2_{H^\frac{1}{2}(0,1)}+\int_0^1 \frac{|u(x)|^2}{x}\dd x\right)^\frac{1}{2},$$
$$||u||_{{^0}H^\frac{1}{2}(0,1)}=\left(||u||^2_{H^\frac{1}{2}(0,1)}+\int_0^1 \frac{|u(x)|^2}{1-x}\dd x\right)^\frac{1}{2}.$$
\end{prop}
Let us give an important example.
\begin{remark}\cite[Example 2.1]{Yama:2022}
For $\beta >\al - \frac{1}{2}$ we have 
$$x^\beta \in {_0}H^\al(0,1).$$
\end{remark}

The following result is of fundamental importance in our further approach.
\begin{prop}\label{domOpFrac}\cite[Example 2.1]{maxreg}, \cite{GLY}.
For $\al \in [0,1]$ the operators $I^\al:L^2(0,1)\rightarrow {_0}H^\al (0,1)$ and $\partial^\al:{_0}H^\al (0,1)\rightarrow L^2(0,1)$ are isomorphism and the following inequalities hold
\begin{equation}
c_\al^{-1}||u||_{{_0}H^\al(0,1)}\leq ||\partial^\al u||_{L^2(0,1)}\leq c_\al||u||_{{_0}H^\al(0,1)}, \quad u\in {_0}H^\al(0,1),
\end{equation}
\begin{equation}
c_\al^{-1}||I^\al f||_{{_0}H^\al(0,1)}\leq ||f||_{L^2(0,1)}\leq c_\al||I^\al f||_{{_0}H^\al(0,1)}, \quad u\in L^2(0,1),
\end{equation}
where $c_\al$ denotes a positive constant dependent on $\al$.

\end{prop}

\begin{coro}\label{domOpFrac2}\cite[Corollary 1]{Rys:2020}
Let $\al,\beta>0$. Then $I^\beta:{_0}H^\al (0,1)\rightarrow {_0}H^{\al+\beta} (0,1)$.
Furthermore, there exists a positive constant $c$ dependent only on $\al,\beta$ such that for every $f \in {}_{0}H^{\al}(0,1)$
\[
\norm{I^{\beta}f}_{{}_{0}H^{\al+\beta}(0,1)} \leq c \norm{f}_{{}_{0}H^{\al}(0,1)}.
\]
\end{coro}

We will also make an extensive use of the following local regularity result.

\begin{lemma}\label{local}\cite[Lemma 4]{Rys:2020}
Let $f\in {}_{0}H^{\al}(0,1)$ for $\al \in (0,1)$ and $\p^{\al}f \in H^{\beta}_{loc}(0,1)$ for $\beta \in (\frac{1}{2},1]$. Then $f \in H^{\beta+\al}_{loc}(0,1)$ and for every $0<\delta <\omega<1$ there exists a positive constant $c=c(\delta,\omega,\al,\beta)$ such that
\eqq{
\norm{f}_{H^{\beta+\al}(\delta,\omega)} \leq c(\norm{f}_{{}_{0}H^{\al}(0,\omega)} + \norm{\p^{\al}f}_{H^{\beta}(\frac{\delta}{2},\omega)}).
}{nowelc}
\end{lemma}

We finish this section with some important identities for fractional derivative. 
\begin{prop}\cite[Proposition 6.10]{KuYa}
If $w \in AC[0, 1]$, then for any $\al \in (0,1)$ the following equality holds
\[
\int_{0}^{1} \p^{\al}w(x) \cdot w(x)\dd x = \frac{\al}{4} \int_{0}^{1} \int_{0}^{1} \frac{\abs{w(x)-w(p)}^{2}}{\abs{x-p}^{1+\al}}\dd p \dd x
\]
\[
+\frac{1}{2\Gamma(1-\al)}\int_{0}^{1} [(1-x)^{-\al} + x^{-\al}]\abs{w(x)}^{2} \dd x.
\]
Hence, there exists a positive constant $c$ which depends only on $\al$, such that
\eqq{
\int_{0}^{1} \p^{\al}w(x) \cdot w(x)\dd x  \geq c \norm{w}_{H^{\frac{\al}{2}}(0,1)}^{2}
}{AH}
\end{prop}



\begin{prop}\label{noweg}
Let $\al \in (0,1)$. If $f,g \in AC[0,1]$ and $g \in C^{0,\beta}([0,1])$ for $\beta \in (\al,1)$, then
\[
\p^{\al} (f\cdot g)(x) = g(x)(\p^{\al} f)(x) + \frac{\al}{\Gamma(1-\al)}\izx (x-p)^{-\al-1}(g(x)-g(p))f(p)\dd p.
\]
\end{prop}


\section{Dirichlet boundary value problem}

The strategy to solve the problem \eqref{StefanProblem0} follows the lines to the authors' previous papers \cite{RoRyVe2022} and \cite{Rys:2020}. We firstly solve the moving boundary problem and then we apply the Schauder's fixed point theorem to obtain the existence result for free boundary problem. Before we proceed with the moving boundary problem let us firstly establish the theory concerning autonomous problem.  For this purpose, we will again relay on the semigroup theory.  

\subsection{Autonomous case}

The results from this subsection comes from \cite[Chapter 3.2]{PhDKasia}. 
Let us consider the following boundary value problem 
\begin{equation}\label{Dirichlet}
 \left\{ \begin{array}{ll}
u_{t} - \frac{\p}{\p x} \da u = f & \textrm{ in } (0,1) \times (0,T),\\
u(0,t) = 0,  u(1,t) = 0 & \textrm{ for  } t \in (0,T), \\
u(x,0) = u_{0}(x) & \textrm{ in } (0,1).\\
\end{array} \right.
\end{equation}
We would like to define the operator $A = \poch \da:D \subseteq L^{2}(0,1) \rightarrow L^{2}(0,1)$ on the appropriate domain $D$ in such a way that $A$ is densely defined sectorial operator and thus generates an analytic semigroup. (For theory of semigroups and evolution operators we refer for example to \cite{Lunardi, Pazy}). 

Taking into account \eqref{zam}, Proposition \ref{domOpFrac} and Proposition \ref{LMderiv}, one could propose
$$D=\{u\in H^{1+\al}(0,1):u_x\in {_0}H^{1+\al}(0,1), u(0)=u(1)=0\}.$$
However, due to the "too many boundary conditions" in the definition of $D$, it can be seen that $\frac{\partial}{\partial x}D^\al:D\subseteq L^2(0,1)\rightarrow L^2(0,1)$ is not a generator of an analytic semigroup. In order to preserve $u(0)=u(1) = 0$, we omit the condition $u_x \in {_0}H^{\al}(0,1)$
and consider 
\eqq{ D\left(\poch\da\right):=\widetilde{\mathcal{D}}_\al=\{u=w-w(1)x^\al: w\in {_0}H^{1+\al}(0,1)\},}{domain}
with the norm
$$||u||_{\widetilde{\mathcal{D}}_\al}=||w||_{H^{1+\al}(0,1)},\quad \text{ for } \al\in(0,1)\setminus\left\{\frac{1}{2}\right\},$$
$$||u||_{\widetilde{\mathcal{D}}_\al}=\left(||w||_{H^{\frac{3}{2}}(0,1)}+\int_0^1 \frac{|w_x(x)|^2}{x}\dd x\right)^\frac{1}{2},\quad \text{ for } \al=\frac{1}{2}.$$



Now we will prove the key result of this subsection.
\begin{theo}\label{semigroup}\cite[Theorem 3.10]{PhDKasia}
The operator $\frac{\partial}{\partial x}D^\al:\widetilde{\mathcal{D}}_\al\subseteq L^2(0,1)\rightarrow L^2(0,1)$ is a generator of an analytic semigroup.
\end{theo}

\begin{proof}
 We begin with the analysis of the resolvent.

\begin{lemma}\label{rezoD}
Let $\poch \da:\ddd \rightarrow L^{2}(0,1)$. Then,
for every $\lambda \in \mathbb{C}$ belonging to the sector
$
\vartheta_{\al}:=\{z \in \mathbb{C} \setminus \{0\}:\abs{\arg z} \leq \frac{\pi (\al+1)}{2}\} \cup \{0\}
$
there holds
$R(\lambda E - \poch\da) = \ld.$
\end{lemma}
\begin{proof}
To prove the lemma we fix  $g \in \ld$ and $\lambda$ belonging to $\vartheta_{\al}$. We must prove that there exists $u \in \ddd$ such that
\eqq{
\lambda u - \poch \da u = g.
}{zwyczajned}
We note that if we search for a solution in $\ddd$, then it can be represented in the form $u = w - w(1)x^{\al}$ and $\da u = \da w - w(1)\Gamma(\al+1)$. Since $w \in {}_{0}H^{1+\al}(0,1)$, we have $\da w \in {}_{0}H^{1}(0,1)$ and hence $(\da u)(0) = -w(1)\Gamma(\al+1)$.
At first we will solve the equation (\ref{zwyczajned}) with initial conditions $u(0) = 0$ and $(\da u)(0) = a$ for arbitrary $a \in \mathbb{C}$ and then we will choose $a$ such that $u(1) = 0$. Let us transform the equation (\ref{zwyczajned}) into integral form. To this end we assume that $u$ which may be written in the form $u = w+\frac{a}{\Gamma(1+\al)}x^{\al}$ where $w \in {}_{0}H^{1+\al}(0,1)$, solves~(\ref{zwyczajned}).
Then, having integrated (\ref{zwyczajned}) we obtain that
\[
\da u = (\da u)(0) + \lambda I u - I g = a+ \lambda I u - I g.
\]
Applying $\ia$ we get
\eqq{
u = a\ia 1+ \lambda I^{\al+1} u - I^{\al+1} g.
}{calkowed}
We note that if we search for a solution such that there exists $w \in {}_{0}H^{1+\al}(0,1)$, such that $u = w+\frac{a}{\Gamma(1+\al)}x^{\al}$, then equation (\ref{zwyczajned}) is equivalent with (\ref{calkowed}). Indeed, it follows from Proposition~\ref{domOpFrac} together with Proposition~\ref{IFpowers}.\\
 We apply the operator $I^{\al+1}$ to (\ref{calkowed}) and we obtain
\[
u= a\ia 1- I^{\al+1}g + \lambda aI^{2\al+1} 1 + \lambda^{2}I^{2(\al+1)}u - \lambda I^{2(\al+1)}g.
\]
Iterating this procedure $n$ times we arrive at
\eqq{
u = a\sum_{k=0}^{n}\lambda^{k} I^{\al+k(\al+1)}1 - \sum_{k=0}^{n}\lambda^{k}I^{(k+1)(\al+1)}g + \lambda^{n+1}I^{(n+1)(\al+1)}u.
}{dond}
We will show, that the last expression tends to zero as $n\rightarrow \infty$.
Indeed, we may note that, since ${}_{0}H^{1+\al}(0,1) \subseteq L^{\infty}(0,1)$ and due to the presence of the $\Gamma$-function in the denominator we have
\[
\abs{\lambda^{n}(I^{n(\al+1)}u)(x)} \leq \norm{u}_{L^{\infty}(0,1)}\frac{\abs{\lambda}^{n} x^{(\al +1)n}}{\Gamma((\al+1)n+1)} \leq \frac{\norm{u}_{L^{\infty}(0,1)}\abs{\lambda}^{n} }{\Gamma((\al+1)n+1)}\rightarrow 0 \textrm{ as } n\rightarrow \infty
\]
for each $\lambda \in \mathbb{C}$ uniformly with respect to $x\in [0,1]$. Thus, passing to the limit with $n$ in (\ref{dond}) we obtain the formula
\eqq{
u = a\sum_{k=0}^{\infty}\lambda^{k} I^{\al+k(\al+1)}1 - \sum_{k=0}^{\infty}\lambda^{k}I^{(k+1)(\al+1)}g.
}{szereg1d}
We note that
\[
\sum_{k=0}^{\infty}\lambda^{k} I^{\al+k(\al+1)}1 = x^{\al}E_{\al+1,\al+1}(\lambda x^{\al+1}),
\]
where $E_{a,b}(\cdot)$ denotes the two parameter Mittag-Leffler function. Furthermore, one may easily show that the second series is uniformly convergent and
\[
\sum_{k=0}^{\infty}\lambda^{k}I^{(k+1)(\al+1)}g = g*x^{\al}E_{\al+1,\al+1}(\lambda x^{\al+1}),
\]
where $*$ denotes the convolution on the positive real line, i.e. $(f*g)(x):=\int_{0}^{x}f(p)g(x-p)\dd p$.
Together, we obtain that function $u$ given by (\ref{szereg1d}) may be equivalently written as
\eqq{
u(x) = ax^{\al}E_{\al+1,\al+1}(\lambda x^{\al+1}) - g*x^{\al}E_{\al+1,\al+1}(\lambda x^{\al+1})
}{initd}
and one may check that $u$ given by the formula (\ref{initd}) is a solution to (\ref{calkowed}) and (\ref{zwyczajned}) with boundary conditions $u(0) = 0$ and $(\da u)(0) = a$. It remains to choose the value $a$ in such a way that $u(1)=0$. For this purpose, we take $x=1$ in (\ref{initd}) and we obtain
\[
u(1) =aE_{\al+1,\al+1}(\lambda) -   (g*y^{\al}E_{\al+1,\al+1}(\lambda y^{\al+1}))(1).
\]
To obtain that $u(1) = 0$ we choose
\[
a = E_{\al+1,\al+1}(\lambda))^{-1} (g*y^{\al}E_{\al+1,\al+1}(\lambda y^{\al+1}))(1).
\]
Note that $a$ is well defined because, taking $\nu=\mu=\al+1$ in \cite[Theorem 4.2.1]{PoSe:2022}, we obtain that $E_{\al+1,\al+1}(\lambda)\neq 0$ for $\lambda$ belonging to the sector $\vartheta_{\alpha}$. Summing up the results we obtain that there exists a solution to (\ref{zwyczajned}) which belongs to $\ddd$ and it is represented by the formula
\[
u(x) =\frac{(g*y^{\al}E_{\al+1,\al+1}(\lambda y^{\al+1}))(1)}{E_{\al+1,\al+1}(\lambda)}x^{\al}E_{\al+1,\al+1}(\lambda x^{\al+1}) - g*x^{\al}E_{\al+1,\al+1}(\lambda x^{\al+1}).
\]
We note that here function $w$ from the definition of $\ddd$ is given by
\[
w(x) =\frac{(g*y^{\al}E_{\al+1,\al+1}(\lambda y^{\al+1}))(1)}{E_{\al+1,\al+1}(\lambda)}x^{\al}\sum_{n=1}^{\infty}\frac{\lambda^{n}x^{(\al+1)n}}{\Gamma((\al+1)n+\al+1)} - g*x^{\al}E_{\al+1,\al+1}(\lambda x^{\al+1}).
\]
In this way we proved the lemma.
\end{proof}
Our next aim is to prove the following.
\begin{lemma}\label{eliptd}
For $u \in \ddd$ we have
\eqq{\Re(-\poch D^{\al}u,u) \geq \ca \norm{u}_{H^{\frac{1+\al}{2}}(0,1)}^{2}}{koerd}
and
\eqq{\abs{(-\poch D^{\al}u,u)} \leq \ba \norm{u}_{H^{\frac{1+\al}{2}}(0,1)}^{2},}{boundd}
where $\ca, \ba$ are positive constants which depends only on $\al$.
\end{lemma}

\begin{proof}
Let us start with (\ref{koerd}). We fix $u \in \ddd$. Since $u(0)=u(1)=0$, we may integrate by parts to obtain
\[
\Re \left(-\poch \da u, u\right) = -\Re \izj (\poch \da u)(x) \cdot \overline{u(x)}\dd x
\]
\[
= \izj \da \Re u(x)\cdot \poch \Re u(x)\dd x + \izj \da \Im u(x)\cdot \poch \Im u(x)\dd x.
\]
We note that $\ddd \subseteq AC[0,1]$ hence we may apply property (\ref{super}) from Proposition \ref{IFpowers} and we get

\[
\Re \left(-\poch \da u, u\right) \hspace{-0.1cm}=\hspace{-0.1cm} \izj \da \Re u(x)\cdot \p^{1-\al}\da \Re u(x)\dd x + \izj \da \Im u(x)\cdot \p^{1-\al}\da \Im u(x)\dd x.
 \]
By the definition of $\ddd$ we know that $\da u \in AC[0,1]$, hence we are allowed to apply inequality (\ref{AH}) with $w=\da \Re u$ and $w= \da \Im u$ to obtain
\[
\Re \left(-\poch \da u, u\right)  \geq \ca \norm{\da u}_{H^{\frac{1-\al}{2}}(0,1)}^{2} \geq \ca \norm{\p^{\frac{1-\al}{2}}\da u}_{L^{2}(0,1)}^{2}
\]
\[
 = \ca \norm{D^{\frac{1+\al}{2}}\ u}_{L^{2}(0,1)}^{2} = \ca \norm{\p^{\frac{1+\al}{2}}\ u}_{L^{2}(0,1)}^{2} \geq \ca \norm{u}_{H^{\frac{1+\al}{2}}(0,1)}.
\]
Here the first equality follows from (\ref{super}), while
the second and third  come from Proposition \ref{domOpFrac} and the fact that $u$ vanishes at zero.
It remains to show (\ref{boundd}). Using superposition property of fractional operators and the definition of the Riemann-Liouville fractional derivative, similarly as in \cite{Rys:2020-2} we arrive at the following sequence of identities
\[
 \poch \da u = \p^{\frac{1+\al}{2}}\p^{\frac{1-\al}{2}} \da u
= \poch I^{\frac{1-\al}{2}} D^{\frac{1+\al}{2}}u.
\]
Then, we apply integration by parts formula ( we note that the boundary terms vanish due to $u(0)=u(1)=0$). and Fubini's theorem to get
\[
\izj \poch \da u \cdot \bar{u}\dd x = - \izj I^{\frac{1-\al}{2}} D^{\frac{1+\al}{2}}u\cdot \bar{u}_{x}\dd x = - \izj D^{\frac{1+\al}{2}}u\cdot I^{\frac{1-\al}{2}}_{-} \bar{u}_{x}\dd x,
\]
where for $\beta \in (0,1)$,  by $I^{\beta}_{-}$ we denote the formal adjoint to $I^{\beta}$, which is given by the formula $(I^{\beta}_{-}u)(x) = \frac{1}{\Gamma(\beta)}\int_{x}^{1}(p-x)^{\beta-1}u(p)\dd p$. Similarly we denote $\partial_{-}^{\beta}u = -\frac{d}{dx}I^{1-\beta}_{-}u$, $D_{-}^{\beta}u = -\frac{d}{dx}I^{1-\beta}_{-}[u(x)-u(1)].$ 

Finally, we get
\[
\abs{\izj \poch \da u \cdot \bar{u}dx}\leq \norm{D^{\frac{1+\al}{2}}u}_{L^{2}(0,1)}\norm{D^{\frac{1+\al}{2}}_{-}\bar{u}}_{L^{2}(0,1)}
\]
\[
= \norm{\p^{\frac{1+\al}{2}}u}_{L^{2}(0,1)}\norm{\p^{\frac{1+\al}{2}}_{-}\bar{u}}_{L^{2}(0,1)}
\leq b_{\al}\norm{u}_{H^{\frac{1+\al}{2}}}^{2}.
\]
We note that here we again used the fact that $u$ vanishes at the boundary and we applied Proposition \ref{domOpFrac}.
\end{proof}
Having proven Lemma \ref{rezoD} and Lemma \ref{eliptd} the proof of Theorem \ref{semigroup} follows from the standard argument (see for example \cite[Chapter 7, Theorem 2.7.]{Pazy}).
\end{proof}

Due to Theorem \ref{semigroup}, and  \cite[Theorem 3.4 Chapter 3.2.1]{Yagi}, we obtain the following existence result for problem \eqref{Dirichlet}.

\begin{coro}\label{uniqsolDirSquare}
Let $f\in C^{o,\nu}([0,T];L^2(0,1))$ for $\nu\in(0,1)$. Then, there exists a unique solution to \eqref{Dirichlet} in $C([0,T];L^2(0,1)) \cap C((0,T];\widetilde{\mathcal{D}}_\al) \cap C^{1}((0,T];L^2(0,1)).$
\end{coro}

\subsection{The moving boundary problem}

Now we are able to deal with the moving boundary problem, i.e.
\begin{equation}\label{MBP}
\begin{array}{lll}
(i) & u_t-\frac{\partial}{\partial x}D^\al u=0 & \text{ in } Q_{s,T},\\
(ii) & u(0,t)=0, u(s(t),t)=0, &   t\in (0,T),\\
(iii) & u(x,0)=u_0(x)\geq 0, &  0<x<s(0)=b,\\
\end{array}
\end{equation}
with  given functions $s:[0,T]\rightarrow \bbR$ , and $u_0 \in {_0}H^{1+\al}(0,b)$ such that $u_0(0)=u_0(b)=0$. We assume that
\begin{equation}\label{assums}
s\in C^{0,1}[0,T], s(0)=b, \text{ and there exists } M>0 \text{ such that } 0\leq\dot{s}(t)\leq M \, \text{ a.e. } t\in [0,T].
\end{equation}

Let us make the following basic but useful observation.
\begin{remark}\label{acotS}
Let $s$ be a function satisfying \eqref{assums}. Then,
$$b\leq s(t)\leq b+MT,\quad \forall t\in[0,T].$$
\end{remark}

Similarly as in \cite{RoRyVe2022} and \cite{Rys:2020}, we will solve \eqref{MBP} applying the theory of evolution operators. To this end, we pass to a cylindrical domain applying the standard substitution $p=\frac{x}{s(t)}$ to obtain
\begin{equation}\label{StefanProblemwithsCil}
\begin{array}{lll}
(i) & v_t-x\frac{\dot{s}(t)}{s(t)}v_x - \frac{1}{s^{1+\al}(t)} \frac{\partial}{\partial x}D^\al v=0, & 0<x<1, 0<t<T,\\
(ii) & v(0,t)=0, v(1,t)=0, &   t\in (0,T),\\
(iii) & v(x,0)=v_0(x), &  0<x<1,\\
\end{array}
\end{equation}
where $v_0(x)=u_0(bx)$. Then, we define
\eqq{A(t):=\frac{1}{s^{1+\al}(t)} \frac{\partial}{\partial x}D^\al, \hd A(t):\ddd \subseteq L^2(0,1)\rightarrow L^2(0,1) \m{ for every } t\in [0,T], }{Adef}
we denote $F(x,t,v):=x\frac{\dot{s}(t)}{s(t)}v_x$ and we may rewrite \eqref{StefanProblemwithsCil} as follows:
\begin{equation}\label{StefanProblemwithsCilEvol}
\begin{array}{lll}
(i) & v_t=A(t)v+F(x,t,v), & 0<x<1, 0<t<T,\\
(ii) & v(0,t)=0, v(1,t)=0, &   t\in (0,T),\\
(iii) & v(x,0)=v_0(x), &  0<x<1.\\
\end{array}
\end{equation}



\begin{lemma}\label{equivL2Dal}
For all $v\in \widetilde{\mathcal{D}}_\al$, there exists constants $c=c(\al,b,M,T)$ and $C=C(\al,b,M,T)$ such that
$$c||v||_{\widetilde{\mathcal{D}}_\al}\leq ||A(t)v||_{L^2(0,1)}\leq C ||v||_{\widetilde{\mathcal{D}}_\al},\quad \forall t\in [0,T].$$
\end{lemma}

\begin{proof}
The proof is analogous to the proof of \cite[Lemma 2]{RoRyVe2022}, one just have to use \eqref{assums} and that $\poch D^{\alpha} x^{\alpha}~=~0$.
\end{proof}

Since we deal with nonautonomous problem \eqref{StefanProblemwithsCil}, we recall the definition of evolution operator and the basic solvability result.

\begin{defi}\cite[Definition 6.0.1]{Lunardi}\label{defewo}
Let $X$ be a Banach space, $T>0$. A family of linear bounded operators $\{G(t,\sigma): 0 \leq \sigma\leq t\leq T\}$ is said to be an evolution operator for the problem
\[
u'(t) = A(t)u + f(t),  0<t \leq T, u(0) = \uz,
\]
where $A(\cdot)$ denotes a family of sectorial operators with common domains, i.e. $D(A(t))\equiv D$ for every $t \in [0,T]$, if
\begin{enumerate}
  \item $G(t,\sigma)G(\sigma,r) = G(t,r),  G(\sigma,\sigma) = E,  \text{ for } 0\leq r\leq \sigma \leq t \leq T$,
  \item $G(t,\sigma) \in B(X,D)  \text{ for } 0\leq \sigma \leq t\leq T$,
  \item $t\mapsto G(t,\sigma)$ is differentiable in $(\sigma,T)$ with values in $B(X)$ and
      \[
      \frac{\p}{\p t}G(t,\sigma) = A(t)G(t,\sigma) \hd  \text{ for } \hd 0 \leq \sigma<t \leq T.
      \]
\end{enumerate}
\end{defi}

\begin{remark}
We note that, if $v$ is a solution to \eqref{StefanProblemwithsCilEvol}$(i)$,\eqref{StefanProblemwithsCilEvol}$(iii)$, that is, $v$ is a solution to
\begin{equation}
v'(\cdot,t) = A(t)v(\cdot,t) + F(\cdot,t,v),  0<t \leq T, v(\cdot,0) = v_0(x),
\end{equation}
and additionally $v(\cdot,t)\in \widetilde{\mathcal{D}}_\al$, $\forall t\in [0,T]$, then by definition of $\widetilde{\mathcal{D}}_\al$ it follows that $v(0,t)=v(1,t)=0$, $\forall t\in [0,T]$, thus $v$ is a solution to \eqref{StefanProblemwithsCil}.
\end{remark}

\begin{theo}\cite[Chapter 6]{Lunardi}\label{atew}
Let $D$ be a Banach space continuously embedded into $X$ and let $T>0$, $a \in (0,1)$. If for $0\leq t \leq T$ $A(t):D(A(t))\rightarrow X$ satisfies that
\begin{enumerate}
\item for every $t \in [0,T]$ $A(t)$ is sectorial and $D(A(t))\equiv D$,
\item $t\mapsto A(t) \in C^{0,a}([0,T];B(D,X))$,
\end{enumerate}
then there exists a family of evolution operators for $A(t)$ given by Definition \ref{defewo}.
\end{theo}

\begin{remark}
By Theorem \ref{semigroup} and \eqref{assums}, it follows that $A(t)$ verifies the hyphotesis of Theorem \ref{atew} for $0\leq t \leq T$. Then, $A(t):\widetilde{\mathcal{D}}_\al\rightarrow L^2(0,1)$ generates a family of evolution operators.
\end{remark}

\begin{prop}\cite[Corollary 6.1.6.(i), (iii)]{Lunardi}\label{upto}
Let $A(t)$ satisfies the assumptions of Theorem \ref{atew}.
If $\uz \in X$, then $G(t,0)\uz \in C([0,T];X)\cap C^{1}((0,T];X)\cap C((0,T];D)$.
Furthermore, if $\uz \in D$, then $G(t,0)\uz \in C^{1}([0,T];X)\cap C([0,T];D)$ and
$\frac{\p}{\p t}G(t,0)\uz = A(t)G(t,0)\uz \text{ for every } 0 \leq t \leq T$.
\end{prop}
In order to develop the theory of non-homogenous problems we introduce the notion of a mild solution.

\subsection{Existence of a mild solution}

Our next purpose is to introduce the concept of mild solution associated with the problem \eqref{StefanProblemwithsCil}.
Recall that, if 
\begin{equation}\label{voc}
u'(t) = A(t)u(t) + f(t), \hd \sigma<t\leq T,  \hd u(\sigma) = u_{\sigma},
\end{equation}
where $A(t)$ satisfies the assumptions of Theorem \ref{atew}, and $\{G(t,\tau): \sigma \leq \tau\leq t\leq T\}$ the family of evolution operators generated by $A(t)$, then for every $f \in L^{1}(\sigma,T;X)$ and $u_{\sigma} \in X$, the function $u$ defined by the formula
\begin{equation}\label{mild}
u(t) = G(t,\sigma)u_{\sigma} + \int_{\sigma}^{t}G(t,\tau)f(\tau)d\tau
\end{equation}
is called a mild solution to (\ref{voc}).

\begin{prop}\cite[Corollary 6.2.4.]{Lunardi} \label{comi}
Let $f \in C((\sigma,T];X)\cap L^{1}(\sigma,T;X)$, $u_{\sigma} \in \overline{D}$. If problem (\ref{voc}) has a solution belonging to
$C^{1}((0,T];X) \cap C((0,T];D) \cap C([0,T];X)$ so that (\ref{voc}) is satisfied for each $t \in (0,T]$, then $u$ is given by (\ref{mild}).
\end{prop}

Returning to \eqref{StefanProblemwithsCil}, we will use \eqref{mild} to define a mild solution to this problem.
\begin{defi}\label{mildsolstand}
Let us denote by $\{G(t,\tau): \sigma \leq \tau\leq t\leq T\}$ the family of evolution operators generated by $A(t)$ defined in (\ref{Adef}).
We say that $v$ is a mild solution to \eqref{StefanProblemwithsCil} if it verifies
\begin{equation}\label{mildSw}
v(x,t)=G(t,0)v_0+\int_0^t G(t,\sigma)x\frac{\dot{s}(\sigma)}{s(\sigma)}v_x(x,\sigma)\dd\sigma.
\end{equation}
\end{defi}

Note that $v$ is involved in the right hand side of \eqref{mildSw} too, and if we define $f(x,t)=x\frac{\dot{s}(t)}{s(t)}v_x(x,t)$, then $v$ is a mild solution to \eqref{StefanProblemwithsCil} if it verifies
$$v(x,t)=G(t,0)v_0+\int_0^t G(t,\sigma)f(x,\sigma)\dd\sigma.$$
Since $f$ depends on $v_x$, we will prove that there exists a mild solution to \eqref{StefanProblemwithsCil}, by using a fix point theorem. For this purpose, we need the following estimates for the interpolation spaces.
\begin{prop}\cite[Corollary 6.1.8]{Lunardi}\label{estgen}
Let $\{G(t,\sigma): 0 \leq \sigma\leq t\leq T\}$ be a family of evolution operators generated by $A(t):D\rightarrow X$. Then, for every $g\in L^2(0,1)$ we have
\begin{equation}\label{r0o}
\norm{G(t,\sigma)g}_{X} \leq c\norm{g}_{X}.
\end{equation}
If $g\in [X,D]_{\delta}$, then for any $0\leq \sigma < t \leq T$ there exists  positive constant $c=c(\theta,\delta,T)$ which is a continuous increasing function of $T$ such that, for every $0<\delta<1$
\begin{equation}\label{r1o}
\norm{G(t,\sigma)g}_{D} \leq \frac{c}{(t-\sigma)^{1-\delta}}\norm{g}_{[X,D]_{\delta}}.
\end{equation}
Moreover, for any $0 \leq \delta < \theta<1$, we have
\eqq{
\norm{G(t,\sigma)g}_{[X,D]_{\theta}} \leq \frac{c}{(t-\sigma)^{\theta-\delta}}\norm{g}_{[X,D]_{\delta}}
}{r4o}
and for $\theta \in (0,1)$, $\delta \in (0,1]$, $\theta<\delta$
\eqq{
\norm{A(t)G(t,\sigma)g}_{[X,D]_{\theta}} \leq \frac{c}{(t-\sigma)^{1+\theta-\delta}}\norm{g}_{[X,D]_{\delta}}.
}{r2o}
Furthermore, for every $0\leq \theta <\delta <1$ and $0\leq \sigma <t \leq T$
\eqq{
\norm{G(t,\sigma)g-G(r,\sigma)g}_{[X,D]_{\theta}}\leq \frac{c}{(t-r)^{\theta-\delta}}
\norm{g}_{[X,D]_{\delta}}.
}{r5o}
Finally, for every $0\leq \sigma < r < t \leq T$
\eqq{
\norm{A(t)G(t,\sigma)g-A(r)G(r,\sigma)g}_{X}\leq c\left(\frac{(t-r)^{a}}{(r-\sigma)^{1-\delta}}+\frac{1}{(r-\sigma)^{1-\delta}}-\frac{1}{(t-\sigma)^{1-\delta}}\right)
\norm{g}_{[X,D]_{\delta}},
}{r3o}
where $a \in (0,1)$ comes from Theorem \ref{atew}.
The constant $c>0$ depends only on $\al,\theta,\delta,T$, and the constants $b,M$ comes from \eqref{assums}. Moreover, $T\mapsto c(\al,\theta,\delta,b,M,T)$ is an increasing function.
\end{prop}

Now we present two substantial estimates, which will be frequently used in the paper. 

\begin{lemma}\label{acotxvxHal}
For every $v\in \widetilde{\mathcal{D}}_\al$, there exists a constant $c>0$ such that
$$||xv_x||_{{_0}H^\al(0,1)}\leq c ||v||_{\widetilde{\mathcal{D}}_\al}.$$
\end{lemma}

\begin{proof}
Since $v\in \widetilde{\mathcal{D}}_\al$, there exists $w\in {_0}H^{1+\al}(0,1)$ such that $v(x)=w-w(1)x^\al$. Then, $xv_x=xw_x-\al w(1)x^\al$, and by Corollary \ref{domOpFrac2} and Sobolev embedding, it follows that
\begin{equation}
\begin{split}
||xv_x||_{{_0}H^{\al}(0,1)}&\leq c||w_x||_{{_0}H^{\al}(0,1)}+\al |w(1)| ||x^\al||_{{_0}H^{\al}(0,1)}\leq c(||v||_{\widetilde{\mathcal{D}}_\al}+||w||_{C[0,1]})\\
&\leq c(||v||_{\widetilde{\mathcal{D}}_\al}+||w||_{{_0}H^{1+\al}(0,1)})\leq c||v||_{\widetilde{\mathcal{D}}_\al}.
\end{split}
\end{equation}
and the claim holds.
\end{proof}

\begin{remark}\label{embeddings1}
For all $ \theta\in(0,1)$, $H_0^{(1+\al)\theta}(0,1)\hookrightarrow [L^2(0,1),\widetilde{\mathcal{D}}_\al]_
{\theta}$ continuously. In fact, $H_0^{(1+\al)\theta}(0,1)=[L^2(0,1),H_0^{1+\al}(0,1)]_\theta$,
and $H_0^{1+\al}(0,1)\hookrightarrow \widetilde{\mathcal{D}}_\al$ continuously. Moreover, for all $\delta < \frac{1}{2(1+\al)}$ we have that $H^{(1+\al)\delta}(0,1)=H_0^{(1+\al)\delta}(0,1)$. Then, for every $ (1+\al)\delta\in \left(0,\frac{1}{2}\right)$ there holds  
$H^{(1+\al)\delta}(0,1)\hookrightarrow [L^2(0,1),\widetilde{\mathcal{D}}_\al]_\delta$  and
\begin{equation}\label{acotIntHal}
||u||_{[L^2(0,1),\widetilde{\mathcal{D}}_\al]_\delta} \leq c||u||_{H^{(1+\al)\delta}}.
\end{equation}
\end{remark}

We prove now existence of a mild solution to \eqref{StefanProblemwithsCil}.

\begin{theo}\label{ExistUnicMildSol}
Let $v_0\in \widetilde{\mathcal{D}}_\al$ be. Then, there exists a unique solution to \eqref{mildSw} belonging to $C([0,T];\widetilde{\mathcal{D}}_\al)$.
\end{theo}

\begin{proof}
The proof is analogous to that of \cite[Theorem 2]{Rys:2020}.
We define an operator $P$ by
\begin{equation}
(Pv)(x,t)=G(t,0)v_0+\int_0^t G(t,\sigma)x\frac{\dot{s}(\sigma)}{s(\sigma)}v_x(x,\sigma)\dd\sigma=G(t,0)v_0+\int_0^t G(t,\sigma)f(x,\sigma)\dd\sigma.
\end{equation}
We will apply the Banach fixed point theorem to $P$. To this end, we will show that $P:C([0,T];\widetilde{\mathcal{D}}_\al)\rightarrow C([0,T];\widetilde{\mathcal{D}}_\al)$. We note that, by \cite[Corollary 6.1.6 (i), (iii)]{Lunardi}, we obtain that $G(t,0)v_0\in C([0,T];\widetilde{\mathcal{D}}_\al)$. 

Let $v\in C([0,T];\widetilde{\mathcal{D}}_\al)$. Then, for every $0<\theta<\frac{\min\{\al,1/2\}}{1+\al}$, by \eqref{r1o}, \eqref{acotIntHal} and Lemma \ref{acotxvxHal}, it follows that
\begin{equation}\label{acotfixedpoint}
\begin{split}
\left|\left|\int_0^t G(t,\sigma)f(\cdot,\sigma)\dd\sigma\right|\right|_{\widetilde{\mathcal{D}}_\al}
&\leq \int_0^t c(t-\sigma)^{\theta-1}\left|\left|f(\cdot,\sigma)\right|\right|_{[L^2(0,1);\widetilde{\mathcal{D}}_\al]_\theta}\dd\sigma\\
&\leq \int_0^t c(t-\sigma)^{\theta-1}\left|\left|f(\cdot,\sigma)\right|\right|_{{_0}H^{(1+\al)\theta}(0,1)}\dd\sigma\\
&\leq c\left|\left|x\frac{\dot{s}(\cdot)}{s(\cdot)}v_x\right|\right|_{L^\infty((0,T),{_0}H^{\al}(0,1))} t^{\theta}.
\end{split}
\end{equation}

Then $\int_0^t G(t,\sigma)f(\cdot,\sigma)\dd\sigma$ is continuous at $t=0$ in the norm $||\cdot||_{\widetilde{\mathcal{D}}_\al}$.

On the other hand, for $0<\tau\leq t \leq T$, applying, \eqref{r0o}, \eqref{r4o} and \eqref{r5o} it follows that
\begin{equation}
\begin{split}
&\left|\left|\int_0^t G(t,\sigma)f(\cdot,\sigma)\dd\sigma-\int_0^\tau G(\tau,\sigma)f(\cdot,\sigma)\dd\sigma\right|\right|_{L^2(0,1)}\\
&\leq \int_0^\tau \left|\left|G(t,\sigma)f(\cdot,\sigma)-G(\tau,\sigma)f(\cdot,\sigma)\right|\right|_{L^2(0,1)}\dd\sigma+\int_\tau^t \left|\left|G(t,\sigma)f(\cdot,\sigma)\right|\right|_{L^2(0,1)}\dd\sigma\\
&\leq \int_0^t c(t-\tau)^{\theta}\left|\left|f(\cdot,\sigma)\right|\right|_{[L^2(0,1);\widetilde{\mathcal{D}}_\al]_\theta}\dd\sigma+c\int_\tau^t \left|\left|f(\cdot,\sigma)\right|\right|_{L^2(0,1)}\dd\sigma\\
\end{split}
\end{equation}

Thus, for every $0<\theta<\frac{\min\{\al,1/2\}}{1+\al}$, by Remark \ref{embeddings1} and Lemma \ref{acotxvxHal} 
\begin{equation}
\begin{split}
&\left|\left|\int_0^t G(t,\sigma)f(\cdot,\sigma)\dd\sigma-\int_0^\tau G(\tau,\sigma)f(\cdot,\sigma)\dd\sigma\right|\right|_{L^2(0,1)}\\
&\leq c(t-\tau)^{\theta}\int_0^\tau \left|\left|f(\cdot,\sigma)\right|\right|_{{_0}H^{(1+\al)\theta}(0,1)}\dd\sigma+c\int_\tau^t \left|\left|f(\cdot,\sigma)\right|\right|_{{_0}H^\al(0,1)}\dd\sigma\\
&\leq c\left|\left|f\right|\right|_{L^\infty((0,T),{_0}H^{\al}(0,1))}\left(\tau(t-\tau)^\theta+(t-\tau)\right),
\end{split}
\end{equation}
and the last term tends to zero when $\tau\rightarrow t$. Then $\int_0^t G(t,\sigma)f(\cdot,\sigma)\dd \sigma\in C((0,T],L^2(0,1))$. By Lemma \ref{equivL2Dal}, we only need to prove that $A(t)\int_0^t G(t,\sigma)f(\cdot,\sigma)\dd \sigma\in C((0,T],L^2(0,1))$.
Let $0<\tau\leq t \leq T$ and $0<\theta<\frac{\min\{\al,1/2\}}{1+\al}$. Then, applying first \eqref{r2o} and \eqref{r3o}, and then Remark \ref{embeddings1} and Lemma \ref{acotxvxHal}, it follows that
\begin{equation}
\begin{split}
&\left|\left|A(t)\int_0^t G(t,\sigma)f(\cdot,\sigma)\dd\sigma-A(\tau)\int_0^\tau G(\tau,\sigma)f(\cdot,\sigma)\dd\sigma\right|\right|_{L^2(0,1)}\\
&\leq \int_0^\tau \left|\left|A(t)G(t,\sigma)f(\cdot,\sigma)-A(\tau)G(\tau,\sigma)f(\cdot,\sigma)\right|\right|_{L^2(0,1)}\dd\sigma+\int_\tau^t \left|\left|A(t)G(t,\sigma)f(\cdot,\sigma)\right|\right|_{L^2(0,1)}\dd\sigma\\
&\leq \int_0^t \left(\frac{(t-\tau)^{a}}{(\tau-\sigma)^{1-\theta}}+\frac{1}{(\tau-\sigma)^{1-\theta}}-\frac{1}{(t-\sigma)^{1-\theta}}\right)\left|\left|f(\cdot,\sigma)\right|\right|_{[L^2(0,1);\widetilde{\mathcal{D}}_\al]_\theta}\dd\sigma+c\int_\tau^t (t-\sigma)^{\theta-1}\left|\left|f(\cdot,\sigma)\right|\right|_{[L^2(0,1);\widetilde{\mathcal{D}}_\al]_\theta}\dd\sigma\\
&\leq c\left|\left|f\right|\right|_{L^\infty((0,T),{_0}H^{\al}(0,1))}\left(\tau^\theta(t-\tau)^a+\tau^\theta -t^\theta +(t-\tau)^\theta\right)\\
\end{split}
\end{equation}
and the last term tends to zero when $\tau\rightarrow t$. Then, we deduce $\int_0^t G(t,\sigma)f(\cdot,\sigma)\dd \sigma\in C([0,T],\widetilde{\mathcal{D}}_\al)$, and we conclude $Pv\in C([0,T],\widetilde{\mathcal{D}}_\al)$.

Now, we will prove that $P$ is a contraction on $C([0,T_1],\widetilde{\mathcal{D}}_\al)$ for $T_1$ small enough. Let $v^1,v^2\in C([0,T_1],\widetilde{\mathcal{D}}_\al)$. By definition of $P$, we have
\begin{equation}\label{diffPw}
(P v^1 -P v^2)(x,t)=\int_0^t G(t,\sigma)x\frac{\dot{s}(t)}{s(t)}(v^1_x-v^2_x)(x,\sigma)\dd\sigma.
\end{equation}

Analogously to \eqref{acotfixedpoint},
\begin{equation}
\begin{split}
\left|\left|\int_0^t G(t,\sigma)x\frac{\dot{s}(t)}{s(t)}(v^1_x-v^2_x)(x,\sigma)\dd\sigma\right|\right|_{C([0,T_1],\widetilde{\mathcal{D}}_\al)}
&\leq \sup_{t\in [0,T_1]} \int_0^t c(t-\sigma)^{\theta-1}\left|\left|x(v^1_x-v^2_x)(x,\sigma)\right|\right|_{{_0}H^{\al}(0,1)}\dd\sigma
\end{split}
\end{equation}
and hence, by Lemma \ref{acotxvxHal},
\begin{equation}
\begin{split}
\left|\left|\int_0^t G(t,\sigma)x\frac{\dot{s}(t)}{s(t)}(v^1_x-v^2_x)(x,\sigma)\dd\sigma\right|\right|_{C([0,T_1],\widetilde{\mathcal{D}}_\al)}&\leq \sup_{t\in [0,T_1]} c\left|\left|v^1-v^2\right|\right|_{C([0,T_1],\widetilde{\mathcal{D}}_\al)}\int_0^t (t-\sigma)^{\theta-1}\dd\sigma\\
&= \frac{c T_1^\theta}{\theta}\left|\left|v^1-v^2\right|\right|_{C([0,T_1],\widetilde{\mathcal{D}}_\al)},
\end{split}
\end{equation}
thus, $P$ is a contraction on $C([0,T_1],\widetilde{\mathcal{D}}_\al)$ whenever $T_1<\left(\frac{\theta}{c}\right)^\frac{1}{\theta}$.
We may extend the solution on the whole interval $[0, T]$ applying the standard argument.

\end{proof}

\subsection{Regularity of the mild solution}

In this section, we improve the regularity of the mild solution obtained in Theorem \ref{ExistUnicMildSol}. We follow the ideas introduced in \cite{Rys:2020} that were also applied in \cite[Section 3.3]{RoRyVe2022}. The main novelty here, is that we deal with the interpolation space $[L^2(0,1),\widetilde{\mathcal{D}}_\al]_\theta$, which characterization is not so immediate as in  previous results (\cite[Equation (3.17)]{RoRyVe2022}. Since we are particularly interested in the interior regularity of the solution, we establish the continuous embedding of the space of restrictions of $[L^2(0,1),\widetilde{\mathcal{D}}_\al]_\theta$ to the interval $(\ve,\omega)$ into $H^{\frac{\theta}{1+\al}}(\ve,\omega)$, $0<\ve <\omega < 1$,  which is presented in the next subsection. The regularity results are given then in subsection~\ref{subsecReg}.

\subsubsection{A relation between the interpolation space $[L^2(0,1),\widetilde{\mathcal{D}}_\al]_\theta$ and a local norm in a fractional Sobolev space}

Let us prove the following result.
\begin{theo}\label{IneqInterpSp}
Let $\al\in (0,1)$. Then, for all $0<\ve<\omega<1$, for every $0<\delta<1+\al$ and $u\in [L^2(0,1), \widetilde{\mathcal{D}}_\al]_\frac{\delta}{1+\al}$, it follows that $\tilde{u} \in H^{\delta}(\ve,\omega)$ and there exists a positive constant $c=c(\ve,\omega,\al,\delta)$ such that
\begin{equation}\label{isom}
||\tilde{u}||_{H^{\delta}(\varepsilon,\omega)}\leq c||u||_{[L^2(0,1), \widetilde{\mathcal{D}}_\al]_\frac{\delta}{1+\al}},
\end{equation}
where $\tilde{u}$ is the restriction of $u$ to the interval $(\varepsilon,\omega)$.
\end{theo}

We will follow the approach in \cite{PaKu}. Thus, let us recall some concepts and auxiliary results for the benefit of the reader. 



\begin{defi}\label{deffamilyF}
Let $(X,Y)$ be an interpolation pair. By $\mathcal{F}(X,Y)$ we denote the set of functions $f : S = \{z \in \bbC : 0 \leq Re(z) \leq 1\}\mapsto X+Y$ such that

$\begin{array}{ll}
\mathcal{F}\text{-}(i) & f \text{ is bounded and continuous in } S,\\
\mathcal{F}\text{-}(ii) & f \text{ is analytic in } S^\circ,\\
\mathcal{F}\text{-}(iii) & f(it)\in X \text{ and }f(1+it)\in Y \text{ for all } t\in \bbR.\\
\mathcal{F}\text{-}(iv) & \text{functions } t\mapsto f(it) \text{ and t } \mapsto f(1+it) \text{ are bounded and continuous with respect to the spaces }\\
 & X \text{ and } Y, \text{ respectively.}\\
\end{array}$


We provide the space $\mathcal{F}(X,Y)$ with the norm
$$||f||_{\mathcal{F}(X,Y)}=\max\{\sup_{t\in\bbR}||f(it)||_{X},\sup_{t\in\bbR}||f(1+it)||_{Y}\}.$$

\end{defi}

\begin{defi}
Let $(X,Y)$ be an interpolation pair. We define the space
$$[X,Y]_\theta = \{ f(\theta) : f \in \mathcal{F}(X,Y)\},$$
with the norm
\begin{equation}\label{norminF}
||\phi||_\theta = \inf\{||f||_{\mathcal{F}(X,Y)} : f (\theta) = \phi\}.
\end{equation}
\end{defi}


Before giving the proof of Theorem \ref{IneqInterpSp} we state the following useful result, where we have considered $X=L^2(0,1)$ and $Y=\widetilde{\mathcal{D}}_\al$.



\begin{lemma}\label{equivnormsL2} Let $u\in L^2(0,1)$ and consider its unique decomposition in $L^2(0,1)+H_0^{1+\al}(0,1)$, i.e., $u=u_p+u_o$ where $u_p\in H_0^{1+\al}(0,1),u_o\in \left(H_0^{1+\al}(0,1)\right)^\perp$. Let the norms on $L^{2}(0,1)$ given by
$||u||_1=||u_o||_{L^2(0,1)}+||u_p||_{H^{1+\al}(0,1)}$ and
$||u||_2=||u||_{L^2(0,1)+\widetilde{\mathcal{D}}_\al}.$
Then the norms $||\cdot||_1$ and $||\cdot||_2$ are equivalent on $L^{2}(0,1)$.
\end{lemma}

\begin{proof}
Let $u\in L^2(0,1)$. Since $H_0^{1+\al}(0,1)\subseteq \widetilde{\mathcal{D}}_\al$, we deduce that
\begin{equation}
\begin{split}
||u||_{L^2(0,1)+\widetilde{\mathcal{D}}_\al}&=\inf\{||u_1||_{L^2(0,1)}+||u_2||_{\widetilde{\mathcal{D}}_\al}: u=u_1+u_2, u_1\in L^2(0,1), u_2\in \widetilde{\mathcal{D}}_\al\}\\
&\leq||u_o||_{L^2(0,1)}+||u_p||_{\widetilde{\mathcal{D}}_\al}=||u_o||_{L^2(0,1)}+||u_p||_{H^{1+\al}(0,1)}.
\end{split}
\end{equation}
Thus, $\forall u\in L^2(0,1)$,
$||u||_2\leq ||u||_1$, and by \cite[Corollary 2.8]{Brezis:2011} we conclude that there exists $c>0$ such that
\begin{equation}\label{equalnorms3}
||u||_1\leq c||u||_2, \quad \forall u\in L^2(0,1),
\end{equation}
and the thesis holds.\end{proof}

\begin{proof}[Proof of Theorem \ref{IneqInterpSp}]

Let $g\in \mathcal{F}(L^2(0,1),\widetilde{\mathcal{D}}_\al) $.
We define a function $G:\mathcal{F}(L^2(0,1),\widetilde{\mathcal{D}}_\al) \mapsto \mathcal{F}(L^2(\varepsilon,\omega),H^{1+\al}(\varepsilon,\omega)) $ by $(Gg)(z)= \widetilde{g}(z)$ such that $\widetilde{g}(z)(x)=\restr{g(z)}{(\varepsilon,\omega)}(x)$.

We note that $H_0^{1+\al}(0,1)$ is a closed subspace of $L^2(0,1)$. Then, by \cite[Theorem 12.4]{Rudin}, $\forall u\in L^2(0,1)$, there exists unique $u_p\in H_0^{1+\al}(0,1),u_o\in \left(H_0^{1+\al}(0,1)\right)^\perp$ such that $u=u_p+u_o$.


We will prove first that $G$ is well defined. That is, we need to prove that $G(g)=\tilde{g}$ verifies the conditions in Definition \ref{deffamilyF}.

\begin{itemize}
\item[$\mathcal{F}$-(i)] Let $z,z_0\in S$ and consider the orthogonal representation for $g$ to obtain,
\begin{equation}\label{eqggo}
\widetilde{g}(z)-\widetilde{g}(z_0)=[\restr{g_o(z)}{(\varepsilon,\omega)}-\restr{g_o(z_0)}{(\varepsilon,\omega)}]+ [\restr{g_p(z)}{(\varepsilon,\omega)}-\restr{g_p(z_0)}{(\varepsilon,\omega)}].
\end{equation}

Then,
\begin{equation}\label{equalnorms1}
\begin{split}
||\widetilde{g}(z)-\widetilde{g}(z_0)||_{L^2(\varepsilon,\omega)+H^{1+\al}(\varepsilon,\omega)}&\leq||g_o(z)-g_o(z_0)||_{L^2(0,1)}+||g_p(z)-g_p(z_0)||_{H^{1+\al}(0,1)}.
\end{split}
\end{equation}

By Lemma \ref{equivnormsL2}, there exists $c>0$, which does not depend on $z$ or $z_0$, such that
\begin{equation}\label{acotnorms12}
||g_o(z)-g_o(z_0)||_{L^2(0,1)}+||g_p(z)-g_p(z_0)||_{H^{1+\al}(0,1)} \leq c ||g(z)-g(z_0)||_{L^2(0,1)+\widetilde{\mathcal{D}}_\al}.
\end{equation}

And from \eqref{equalnorms1} and \eqref{acotnorms12}, we deduce that
$$||\widetilde{g}(z)-\widetilde{g}(z_0)||_{L^2(\varepsilon,\omega)+H^{1+\al}(\varepsilon,\omega)}\leq c ||g(z)-g(z_0)||_{L^2(0,1)+\widetilde{\mathcal{D}}_\al}.$$
Since $g$ is continuous in $S$, we conclude that $\widetilde{g}$ is continuous in $S$.

Anagolously, we can deduce that
\begin{equation}\label{acotnormgtildeg}
||\widetilde{g}(z)||_{L^2(\varepsilon,\omega)+H^{1+\al}(\varepsilon,\omega)}\leq||g(z)||_{L^2(0,1)+\widetilde{\mathcal{D}}_\al}
\end{equation}
which implies that $\widetilde{g}$ is bounded in $S$, since $g$ is bounded in $S$.

\item[$\mathcal{F}$-(ii)] Let $z_0\in S^\circ$. We know that $g$ is analytic in $S^\circ$. Then, there exists $g'(z_0)$ such that
$$\left|\left|\frac{g(z_0+h)-g(z_0)}{h}-g'(z_0)\right|\right|_{L^2(0,1)+\widetilde{\mathcal{D}}_\al}\rightarrow 0,\quad \text{ when } z\rightarrow z_0.$$

A candidate for $\widetilde{g}'(z_0)$ is
$$j(z_0)=\restr{g'(z_0)}{(\varepsilon,\omega)}.$$
In fact, by Lemma \ref{equivnormsL2} there exists $c>0$ such that
\begin{equation}\label{derivgtilde}
\begin{split}
&\left|\left|\frac{\widetilde{g}(z_0+h)-\widetilde{g}(z_0)}{h}-j(z_0)\right|\right|_{L^2(\varepsilon,\omega)+H^{1+\al}(\varepsilon,\omega)}\\
&\leq\left|\left|\frac{\restr{g_o(z_0+h)}{(\varepsilon,\omega)}-\restr{g_o(z_0)}{(\varepsilon,\omega)}}{h}-\restr{g_o'(z_0)}{(\varepsilon,\omega)}\right|\right|_{L^2(\varepsilon,\omega)}\hspace{-0.4cm}+\left|\left|\frac{\restr{g_p(z_0+h)}{(\varepsilon,\omega)}-\restr{g_p(z_0)}{(\varepsilon,\omega)}}{h}-\restr{g_p'(z_0)}{(\varepsilon,\omega)}\right|\right|_{H^{1+\al}(\varepsilon,\omega)}\\
&\leq c\left|\left|\frac{g(z_0+h)-g(z_0)}{h}-g'(z_0)\right|\right|_{L^2(0,1)+\widetilde{\mathcal{D}}_\al}
\end{split}
\end{equation}
and the last expression tends to zero as $z\rightarrow z_0$. Thus, $\widetilde{g}$ is analytic in $S^\circ$.

\item[$\mathcal{F}$-(iii)] For all $g\in \mathcal{F}(L^2(0,1),\widetilde{\mathcal{D}}_\al)$ and $t\in \bbR$, we have
$$\widetilde{g}(it)(x)=\restr{g(it)}{(\varepsilon,\omega)}(x)\in L^2(\varepsilon,\omega)$$
and
$$\widetilde{g}(1+it)(x)=\restr{g(1+it)}{(\varepsilon,\omega)}(x).$$
Since $g(1+it)\in \widetilde{\mathcal{D}}_\al$, there exists a unique $w_{g,t}\in {_0}H^{1+\al}(0,1)$ such that
$$g(1+it)=w_{g,t}-w_{g,t}(1)x^\al.$$
Note that, $x^\al \in H^{1+\al}(\varepsilon,\omega)$ and thus $g(1+it)\in H^{1+\al}(\varepsilon,\omega)$.

\item[$\mathcal{F}$-(iv)] We will only prove that $t\mapsto \widetilde{g}(1+it)$ is bounded and continuous with respect to $H^{1+\al}(\varepsilon,\omega)$, because this part of the statement is more involved.
Let us consider $\widetilde{g}(1+it)$. By hypothesis $g(1+it)\in \widetilde{\mathcal{D}}_\al$, thus 
$g(1+it) =w_{g,t}-w_{g,t}(1)x^\al$,
for $w_{g,t}\in {_0}H^{1+\al}(0,1)$. Hence,
\begin{equation}
\begin{split}
||\widetilde{g}(1+it)-\widetilde{g}(1+it_0)||_{H^{1+\al}(\varepsilon,\omega)}&=||\restr{g(1+it)}{(\varepsilon,\omega)}-\restr{g(1+it_0)}{(\varepsilon,\omega)}||_{H^{1+\al}(\varepsilon,\omega)}\\
&\leq||\restr{w_{g,t}}{(\varepsilon,\omega)}-\restr{w_{g,t_0}}{(\varepsilon,\omega)}||_{H^{1+\al}(\varepsilon,\omega)}+||(w_{g,t_0}(1)-w_{g,t}(1))x^\al||_{H^{1+\al}(\varepsilon,\omega)}\\
&\leq||w_{g,t}-w_{g,t_0}||_{H^{1+\al}(0,1)}+c|w_{g,t_0}(1)-w_{g,t}(1)|\\
&=||g(1+it)-g(1+it_0)||_{\widetilde{\mathcal{D}}_\al}+c|w_{g,t_0}(1)-w_{g,t}(1)|.\\
\end{split}
\end{equation}

Then, since $t\mapsto g(1+it)$ is continuous with respect to $\widetilde{\mathcal{D}}_\al$, we deduce that
$$||g(1+it)-g(1+it_0)||_{\widetilde{\mathcal{D}}_\al}\rightarrow 0,\quad \text{ as } t\rightarrow t_0,$$
or equivalently
$$||w_{g,t}-w_{g,t_0}||_{H^{1+\al}(0,1)}\rightarrow 0,\quad \text{ as } t\rightarrow t_0.$$
Since $H^{1+\al}(0,1)\hookrightarrow C([0,1])$, there exists $c>0$ such that
$$|w_{g,t}(1)-w_{g,t_0}(1)|\leq\sup_{x\in[0,1]}|w_{g,t}(x)-w_{g,t_0}(x)|\leq c ||w_{g,t}-w_{g,t_0}||_{H^{1+\al}(0,1)}\rightarrow 0, \quad \text{ as } t\rightarrow t_0.$$
Thus, we conclude that $t\mapsto \widetilde{g}(1+it)$ is continuous with respect to $H^{1+\al}(\varepsilon,\omega)$.
On the other hand, 
$$ ||\widetilde{g}(1+it)||_{H^{1+\al}(\varepsilon,\omega)}\leq c||\widetilde{g}(1+it)||_{L^2(\varepsilon,\omega) + H^{1+\al}(\varepsilon,\omega)}\leq c||g(1+it)||_{L^2(0,1)+\widetilde{\mathcal{D}}_\al}\leq c||g(1+it)||_{\widetilde{\mathcal{D}}_\al}.$$
Then, since $t\mapsto g(1+it)$ is bounded with respect to $\widetilde{\mathcal{D}}_\al$, we deduce that $t\mapsto \widetilde{g}(1+it)$ is bounded with respect to $H^{1+\al}(\varepsilon,\omega)$.

\end{itemize}

We can conclude now that $\widetilde{g}\in \mathcal{F}(L^2(\varepsilon,\omega),H^{1+\al}(\varepsilon,\omega))$ and hence, $G$ is well defined. \\

We will prove now that $G$ is surjective.
Let $h\in \mathcal{F}(L^2(\varepsilon,\omega),H^{1+\al}(\varepsilon,\omega))$. We define $g:S\mapsto L^2(0,1)+\widetilde{\mathcal{D}}_\al$ by $g(z)=\eta \restr{\rho h(z)}{(0,1)}$, where $\eta$ is a fixed function such that $\eta \in C^\infty_0(0,1)$, $\eta(x)=1$, $\forall x\in (\varepsilon,\omega)$, and $\rho$ is the extension operator defined in \cite[Theorem 8.1]{Lions}. Note that by \cite[Theorem 9.1]{Lions}, $\rho\in \mathcal{L}(H^{1+\al}(\varepsilon,\omega),H^{1+\al}(\bbR))$. For simplicity, we denote $\eta \restr{\rho h(z)}{(0,1)}=\eta \rho h(z)$.

By definition of $g$, we get immediately that $(Gg)(z)=\restr{g(z)}{(\varepsilon,\omega)}=\restr{\eta \rho h(z)}{(\varepsilon,\omega)}=h(z), \quad \forall z\in S$. Thus, we need to prove that $g\in \mathcal{F}(L^2(0,1),\widetilde{\mathcal{D}}_\al)$.

\begin{itemize}
\item[$\mathcal{F}$-(i)] Let $z\in S$. Note that, for every decomposition $h(z)=u_{1,z}+u_{2,z}$ with $u_{1,z}\in L^2(\varepsilon,\omega), u_{2,z}\in  H^{1+\al}(\varepsilon,\omega)$, the following decomposition in $L^2(0,1)+\widetilde{\mathcal{D}}_\al$ for $g(z)$ holds for all $z\in S$
$$g(z)=\eta \rho u_{1,z}(z)+\eta \rho u_{2,z}(z).$$
Additionally, since $\rho \in \mathcal{L}(L^2(\varepsilon,\omega),L^2(\bbR))$
\begin{equation}\label{002}
||\eta \rho u_{1,z}||_{L^2(0,1)}
\leq c||u_{1,z}||_{L^2(\varepsilon,\omega)},
\quad \text{ and } \quad
||\eta \rho u_{2,z}||_{\widetilde{\mathcal{D}}_\al}
\leq c||u_{2,z}||_{H^{1+\al}(\varepsilon,\omega)}.
\end{equation}
Thus, by \eqref{002}, 
there exists $c>0$ such that
\begin{equation}\label{004}
||\eta \rho u_{1,z}||_{L^2(0,1)}+||\eta \rho u_{2,z}||_{\widetilde{\mathcal{D}}_\al}\leq c(||u_{1,z}||_{L^2(\varepsilon,\omega)}+||u_{2,z}||_{H^{1+\al}(\varepsilon,\omega)})
\end{equation}
for every decomposition $h(z)=u_{1,z}+u_{2,z}$. Then
\begin{equation}
\begin{split}
||g(z)||_{L^2(0,1)+\widetilde{\mathcal{D}}_\al}&=\inf\{||w_{1,z}||_{L^2(0,1)}+||w_{2,z}||_{\widetilde{\mathcal{D}}_\al}: g(z)=w_{1,z}+w_{2,z},\,w_{1,z}\in L^2(0,1), w_{2,z}\in \widetilde{\mathcal{D}}_\al\}\\
&\leq  \inf\{||\eta \rho u_{1,z}||_{L^2(0,1)}+||\eta \rho u_{2,z}||_{\widetilde{\mathcal{D}}_\al}: h(z)=u_{1,z}+u_{2,z},\,u_{1,z}\in L^2(\varepsilon,\omega), u_{2,z}\in  H^{1+\al}(\varepsilon,\omega)\}\\
&\leq  c \inf\{||u_{1,z}||_{L^2(\varepsilon,\omega)}+||u_{2,z}||_{H^{1+\al}(\varepsilon,\omega)}: h(z)=u_{1,z}+u_{2,z},\,u_{1,z}\in L^2(\varepsilon,\omega), u_{2,z}\in  H^{1+\al}(\varepsilon,\omega)\}\\
&=c||h(z)||_{L^2(\varepsilon,\omega)+ H^{1+\al}(\varepsilon,\omega)},
\end{split}
\end{equation}
and using the fact that $h\in \mathcal{F}(L^2(\varepsilon,\omega),H^{1+\al}(\varepsilon,\omega))$ we conclude that $g$ is bounded in $S$.

Similarly, for $z,z_0\in S$, we deduce that
\begin{equation}\label{contgS}
||g(z)-g(z_0)||_{L^2(0,1)+\widetilde{\mathcal{D}}_\al}\leq c||h(z)-h(z_0)||_{L^2(\varepsilon,\omega)+ H^{1+\al}(\varepsilon,\omega)},
\end{equation}
where the last expresion tends to $0$ when $z\rightarrow z_0$. Hence, $g$ is continuous in $S$.

\item[$\mathcal{F}$-(ii)] 
A candidate for $g'(z_0)$ is $j(z_0)=\eta \rho (h'(z_0))$.
In fact, from \eqref{004},
\begin{equation}\label{derivgtilde}
\left|\left|\frac{g(z_0+\xi)-g(z_0)}{\xi}-j(z_0)\right|\right|_{L^2(0,1)+\widetilde{\mathcal{D}}_\al}\leq c\left|\left|\frac{ h(z_0+\xi)- h(z_0)}{\xi}-h'(z_0)\right|\right|_{L^2(\varepsilon,\omega)+H^{1+\al}(\varepsilon,\omega)}
\end{equation}
and the last expresion tends to $0$ when $\xi\rightarrow 0$. Hence, $g$ is analytic in $S^\circ$.

\item[$\mathcal{F}$-(iii)] It is clear that $g(it)\in L^2(0,1)$ and  $g(1+it)\in \widetilde{\mathcal{D}}_\al$.

\item[$\mathcal{F}$-(iv)] It is easy to prove that $t\mapsto g(it)$ and $t\mapsto g(1+it)$ are continuous and bounded with respect to $L^2(0,1)$ and $\widetilde{\mathcal{D}}_\al$ respectively, following the proof for $\tilde{g}$.

\end{itemize}

We conclude that $g\in \mathcal{F}(L^2(0,1),\widetilde{\mathcal{D}}_\al)$ and then, $G$ is surjective.

As before, for all $t\in \bbR$ we deduce that
$$||(Gg)(it)||_{L^2(\varepsilon,\omega)}\leq ||g(it)||_{L^2(0,1)}$$
 and
$$||(Gg)(1+it)||_{H^{1+\al}(\varepsilon,\omega)}\leq c||g(1+it)||_{\widetilde{\mathcal{D}}_\al}.$$
Thus, there exists $c>0$ such that
$$\max\{\sup_{t\in\bbR}||Gg(it)||_{L^2(\varepsilon,\omega)},\sup_{t\in\bbR}||Gg(1+it)||_{H^{1+\al}(\varepsilon,\omega)}\}\leq c\max\{\sup_{t\in\bbR}||g(it)||_{L^2(0,1)},\sup_{t\in\bbR}||g(1+it)||_{\widetilde{\mathcal{D}}_\al}\},$$
or equivalently
\begin{equation}\label{acotspacesF}
||Gg||_{\mathcal{F}(L^2(\varepsilon,\omega),H^{1+\al}(\varepsilon,\omega))}\leq c||g||_{\mathcal{F}(L^2(0,1),\widetilde{\mathcal{D}}_\al)}.
\end{equation}

We pass now to the interpolation spaces. By definition,
$$H^{\delta}(\varepsilon,\omega)=[L^2(\varepsilon,\omega),H^{1+\al}(\varepsilon,\omega)]_\frac{\delta}{1+\al}=\left\{f\left(\frac{\delta}{1+\al}\right):f\in \mathcal{F}(L^2(\varepsilon,\omega),H^{1+\al}(\varepsilon,\omega))\right\},$$
$$[L^2(0,1),\widetilde{\mathcal{D}}_\al]_\frac{\delta}{1+\al}=\left\{g\left(\frac{\delta}{1+\al}\right):g\in \mathcal{F}(L^2(0,1),\widetilde{\mathcal{D}}_\al)\right\}.$$

Let $u\in [L^2(0,1),\widetilde{\mathcal{D}}_\al]_\frac{\delta}{1+\al}$. Then, there exists $g\in \mathcal{F}(L^2(0,1),\widetilde{\mathcal{D}}_\al)$ such that $g\left(\frac{\delta}{1+\al}\right)=u$.

We define $\widetilde{u}=\restr{u}{(\varepsilon,\omega)}$. Then
$\widetilde{u}=\restr{u}{(\varepsilon,\omega)}=\restr{g\left(\frac{\delta}{1+\al}\right)}{(\varepsilon,\omega)}=(Gg)\left(\frac{\delta}{1+\al}\right)=:\widetilde{g}\left(\frac{\delta}{1+\al}\right)$,
and by definition of $G$ it holds that $\widetilde{g}\in \mathcal{F}(L^2(\varepsilon,\omega),H^{1+\al}(\varepsilon,\omega))$. Then, $\widetilde{u}=\widetilde{g}\left(\frac{\delta}{1+\al}\right)\in [L^2(\varepsilon,\omega),H^{1+\al}(\varepsilon,\omega)]_\frac{\delta}{1+\al}$, and using \eqref{acotspacesF} it holds that,
\begin{equation}
\begin{split}
||u||_{[L^2(0,1),\widetilde{\mathcal{D}}_\al]_\frac{\delta}{1+\al}}&=\inf\left\{||g||_{\mathcal{F}(L^2(0,1),\widetilde{\mathcal{D}}_\al)}:g\left(\frac{\delta}{1+\al}\right)=u\right\}\\
&\geq c\inf\left\{||Gg||_{\mathcal{F}(L^2(\varepsilon,\omega),H^{1+\al}(\varepsilon,\omega))}:(Gg)\left(\frac{\delta}{1+\al}\right)=\widetilde{u}\right\}\\
&= c\inf\left\{||\widetilde{g}||_{\mathcal{F}(L^2(\varepsilon,\omega),H^{1+\al}(\varepsilon,\omega))}:\widetilde{g}\left(\frac{\delta}{1+\al}\right)=\widetilde{u}\right\}\\
&=c||\widetilde{u}||_{[L^2(\varepsilon,\omega),H^{1+\al}(\varepsilon,\omega)]_\frac{\delta}{1+\al}}\\
&=c||\widetilde{u}||_{H^{\delta}(\varepsilon,\omega)}.
\end{split}
\end{equation}
In this way we complete the proof of Theorem \ref{IneqInterpSp}.
\end{proof}

\subsubsection{Regularity results}\label{subsecReg}

The statements of the following lemmas are analogous to \cite[Lemma 1 and Lemma 2]{Rys:2020}, but the adaptation of the proof for the latter needs to be done carefully, taking into account some details related to estimations when working with the interpolation spaces $[L^2(0,1),\widetilde{\mathcal{D}}_\al]_\theta$.

\begin{lemma}\label{regmildsol}
The mild solution $v$ obtained in Theorem \ref{ExistUnicMildSol} verifies that $v_t \in L^\infty_{loc}(0,T;L^2(0,1))$ and
$$v_t=\frac{1}{s^{1+\al}(t)}\frac{\partial}{\partial x} D^\al v+x\frac{\dot{s}(t)}{s(t)}v_x, \quad \text{ in } L^2(0,1) \text{ for all } t\in (0,T).$$
\end{lemma}

We ommit the proof of Lemma \ref{regmildsol} because it is analogous to the given in \cite[Lemma 1]{Rys:2020}.

Below, we will focus on the case $\al>1/2$, since for smaller values of $\al$, the proofs of the following results are more technical (see \cite[Lemma 3]{Rys:2020} and \cite[Lemma 3.10]{RoRyVe2022}).


\begin{lemma}\label{lemusinewdomain}
Let us assume that $\al \in (\frac{1}{2},1)$ and $u_{\sigma} \in {}_{0}H^{\al}(0,1)$. We denote by $u$ the solution to the equation
\begin{equation}\label{pol}
 \left\{ \begin{array}{ll}
u_{t} =A(t)u & \textrm{ for } 0<x<1,  0\leq \sigma<t<T,\\
u(x,\sigma) = u_{\sigma}(x) & \textrm{ for } 0<x<1, \\
\end{array} \right. \end{equation}
given by the evolution operator generated by the family $A(t)$. Then, for every $0<\gamma<\al$, for every $0<\varepsilon<\omega<1$ there exists a positive constant $c=c(\al,b,M,T,\varepsilon,\omega,\gamma)$, where $b,M$ comes from \eqref{assums}, such that for every $t \in (\sigma,T]$ there holds
\begin{equation}\label{regloc1}
\norm{A(t)u(\cdot,t)}_{H^{\gamma}(\varepsilon,\omega)} \leq c(t-\sigma)^{-\frac{1+\gamma}{1+\al}}\norm{u_{\sigma}}_{{}_{0}H^{\al}(0,1)}.
\end{equation}
\end{lemma}

\begin{proof}
The proof follows the ideas of the one in \cite[Lemma 2]{Rys:2020} but since the characterization in \cite[Eq. (3.17)]{RoRyVe2022} is not applicable to the interpolation spaces under consideration, we will apply Theorem \ref{IneqInterpSp} several times.

Note that since $\al > \frac{1}{2}$ and $u_{\sigma} \in {}_{0}H^{\al}(0,1)$, it follows that $u_{\sigma} \in H^{(1+\al)\delta}(0,1)$, for all $\delta<\frac{1}{2(1+\al)}$. 

On the other hand, by Proposition \ref{comi} to \eqref{pol}, it follows that there exists a unique solution $u\in C([\sigma,T];L^2(0,1))\cap C((\sigma,T];\widetilde{\mathcal{D}}_\al)\cap C^1((\sigma,T];L^2(0,1))$
such that $u(\cdot,t)=G(t,\sigma)u_\sigma$, and from \eqref{r1o} it follows that
$$||u(\cdot,t)||_{\widetilde{\mathcal{D}}_\al}\leq c(t-\sigma)^{\delta-1}||u_\sigma||_{[L^2(0,1),\widetilde{\mathcal{D}}_\al]_\delta}.$$

We recall that the interpolation constant $c$ depends on the parameters of interpolation as well as on $\al, T$ and $b,M$, that is $c=c(\al,b,M,T,\delta)$.
However, as before we neglect this dependency in notation and leave it just in the final results. The same comment applies on the parameters $\varepsilon$ and $\omega$ and we note that the constant $c>0$ may change from line to line.

We fix $0<\varepsilon<\omega<1$ and we set $\omega_{*}= \frac{1+\omega}{2}$. Let us discuss the approximate problem. We choose a sequence $\{\vf^{k}\}\subseteq C_0^\infty(0,1)$ such that
\eqq{ \{\vf^{k}\} \subseteq \widetilde{\mathcal{D}}_\al, \hd \vf^{k}\rightarrow u_{\sigma} \m{ in } {}_{0}H^{\al}(0,\omega_{*}) \m{ and }  \vf^{k}\rightarrow u_{\sigma} \m{ in } H^{\bar{\gamma}}(0,1) \m{ for every } \bar{\gamma} < \frac{1}{2}.}{noweb}

See \cite[Lemma 4.4]{PhDKasia} for the existence proof.

Consider now the following approximate problem
\begin{equation}\label{nowec}
 \left\{ \begin{array}{ll}
u^{k}_{t} =A(t)u^{k} & \textrm{ for }\hd 0<x<1, \hd 0\leq \sigma<t<T,\\
u^{k}(x,\sigma) = \vf^{k}(x) & \textrm{ for } \hd 0<x<1, \\
\end{array} \right. \end{equation}

Then, by \cite[Corollary 6.2.4]{Lunardi} there exists a unique solution $u^k\in C([\sigma,T];L^2(0,1))\cap C((\sigma,T];\widetilde{\mathcal{D}}_\al)\cap C^1((\sigma,T];L^2(0,1))$, and $u^k(\cdot,t)=G(t,\sigma)\varphi^k$. Hence, by Proposition \ref{upto}, it follows that $u^k\in C([\sigma,T];\widetilde{\mathcal{D}}_\al)\cap C^1([\sigma,T];L^2(0,1))$

Let $u,u^k$ be the solutions of \eqref{pol} and \eqref{nowec} respectively.
Thus, by \eqref{r2o} and \eqref{acotIntHal} it follows that, for all $0<\bar{\gamma}<\bar{\gamma_{1}}<\frac{1}{2}$,
\begin{equation}
||A(t)(u-u^k)||_{[L^2(0,1),\widetilde{\mathcal{D}}_\al]_\frac{\overline{\gamma}}{1+\al}}\leq c(t-\sigma)^{-1-\frac{\overline{\gamma}-\overline{\gamma}_1}{1+\al}}||u_\sigma-\vf^k||_{[L^2(0,1),\widetilde{\mathcal{D}}_\al]_\frac{\overline{\gamma_1}}{1+\al}}\leq c(t-\sigma)^{-1-\frac{\overline{\gamma}-\overline{\gamma}_1}{1+\al}}||u_\sigma-\vf^k||_{H^{\overline{\gamma_1}}(0,1)}
\end{equation}
and the last term tends to $0$ when $k\rightarrow \infty$. Hence, for almost all $t \in (\sigma,T]$
$$A(t)u^k(\cdot,t)\rightarrow A(t)u(\cdot,t), \quad \text{ in } [L^2(0,1),\widetilde{\mathcal{D}}_\al]_\frac{\overline{\gamma}}{1+\al} \quad \forall \overline{\gamma}<\frac{1}{2}.$$
Then, taking into account \eqref{assums}, for every $0\leq\bar{\gamma}<\frac{1}{2}$ and almost all $t \in (\sigma,T]$
\eqq{
\poch \da u^{k}(\cdot,t) \rightarrow \poch \da u(\cdot,t)  \hd \m{ in } \hd [L^2(0,1),\widetilde{\mathcal{D}}_\al]_\frac{\overline{\gamma}}{1+\al}.
}{nowed}

Applying \eqref{r2o} first, and subsequently \eqref{acotIntHal} and \eqref{noweb}, for $k$ large enough and every $0\leq \bar{\gamma}<\bar{\gamma_{1}}<\frac{1}{2}$ we have
\begin{equation}\label{nowef}
||A(t)u^k(\cdot,t)||_{[L^2(0,1),\widetilde{\mathcal{D}}_\al]_\frac{\overline{\gamma}}{1+\al}}
\leq c(t-\sigma)^{-1-\frac{\overline{\gamma}-\overline{\gamma}_1}{1+\al}}||\vf^k||_{H^{\overline{\gamma}_1}(0,1)}\leq c(t-\sigma)^{-1-\frac{\overline{\gamma}-\overline{\gamma}_1}{1+\al}}||u_\sigma||_{H^{\overline{\gamma}_1}(0,1)}.
\end{equation}

In order to obtain $k$ - independent  local estimate of $A(t)u^{k}(\cdot,t)$ in more regular space we introduce a smooth non-negative cutoff function $\eta$ such that $\eta\equiv 0$ on $[0,\frac{\varepsilon}{2}]\cup [\omega_{*},1]$, and $\eta\equiv 1$ on $[\varepsilon,\omega]$. At first, we apply $\p^{\al}$ to \eqref{nowec} and multiply the resulting identity by $\eta$. 


Applying Proposition \ref{noweg} we arrive at
\[
\eta \p^{\al} \poch \da u^{k} = -\frac{\al}{\Gamma(1-\al)} \izx (x-p)^{-\al-1}(\eta(x)-\eta(p))\poch \da u^{k}(p)\dd p + \p^{\al} (\poch\da u^{k} \cdot \eta )
\]
\[
= -\frac{\al}{\Gamma(1-\al)} \izx (x-p)^{-\al-1}(\eta(x)-\eta(p))\poch\da u^{k}(p)\dd p + \p^{\al}\poch(\eta \da u^{k}) - \p^{\al}(\eta'\da u^{k}).
\]

Since $u^k(0,t)=0$ for every $t \in (0,T]$, from \eqref{darl} we have
\[
\p^{\al}\poch(\eta \da u^{k})(x,t) = \p^{\al}\poch(\eta \p^{\al}u^{k}(x,t)).
\]

In view of identity (\ref{zam}) we obtain that $\{u^{k}\}$ satisfy the system of equations
\begin{equation}\label{noweh}
 \left\{ \begin{array}{ll}
(\p^{\al} u^{k}\cdot \eta)_{t} - A(t) (\p^{\al} u^{k}\cdot\eta) =F^{k} & \textrm{ for } 0<x<1, \hd 0\leq \sigma<t<T,\\
(\p^{\al} u^{k} \cdot \eta)(\cdot,\sigma) = \p^{\al}\vf^{k}\cdot \eta & \textrm{ for } 0<x<1, \\
\end{array} \right. \end{equation}
where
\[
F^{k}:=\hspace{-0.1 cm} \frac{1}{s^{1+\al}(t)}\hspace{-0.1 cm}\left[\frac{-\al}{\Gamma(1-\al)}\hspace{-0.1 cm} \izx (x-p)^{-\al-1}(\eta(x)-\eta(p))\poch\da u^{k}(p,t)dp - \p^{\al}(\eta'\da u^{k}) \right].
\]

Our aim now is to apply Proposition \ref{comi} to problem \eqref{noweh}. We need to prove that:
\begin{enumerate}[(i)]
\item \label{item1} $F^{k} \in C((\sigma,T];L^{2}(0,1))\cap L^{1}(\sigma,T;L^{2}(0,1))$,

\item \label{item2} $\p^{\al}\vf^k\cdot \eta\in \widetilde{\mathcal{D}}_\al$,

\item \label{item3} $\p^{\al}u^k\cdot \eta\in C([\sigma,T];L^{2}(0,1))\cap C((\sigma,T];\widetilde{\mathcal{D}}_\al)\cap C^1((\sigma,T];L^{2}(0,1))$.
\end{enumerate}

At first, we note that for every $x, p \in [0,1]$, $x\neq p$ we have
\begin{equation}\label{acot0000-0}
\abs{\frac{\eta(x)-\eta(p)}{x-p}} \leq \norm{\eta}_{W^{1,\infty}(0,1)},
\end{equation}
hence,
\[
\frac{1}{\Gamma(1-\al)}\abs{\izx (x-p)^{-\al-1}(\eta(x)-\eta(p))A(t) u^{k}(p,t)\dd p} \leq \norm{\eta}_{W^{1,\infty}(0,1)} I^{1-\al}\abs{A(t) u^{k}(x,t)}.
\]
Since $I^{1-\al}$ is bounded on $L^{2}(0,1)$ we obtain 
\begin{equation}\label{acot0001}
\norm{\frac{\al}{\Gamma(1-\al)}\izx (x-p)^{-\al-1}(\eta(x)-\eta(p))A(t) u^{k}(p,t)\dd p }_{L^{2}(0,1)} \leq c\norm{A(t) u^{k}(\cdot,t)}_{L^{2}(0,1)}.
\end{equation}

By Proposition \ref{domOpFrac} and \eqref{assums}, we may write
\begin{equation}\label{acot0002}
\begin{split}
&\frac{1}{s^{1+\al}(t)}\norm{\p^{\al}(\eta'\da u^{k})}_{L^{2}(0,1)} \leq \frac{c}{s^{1+\al}(t)} \norm{(\eta'\da u^{k})}_{{}_{0}H^{\al}(0,1)}\leq \frac{c}{s^{1+\al}(t)} \norm{(\eta'\da u^{k})}_{{}_{0}H^{1}(0,1)}\\
&\leq \frac{c}{s^{1+\al}(t)} \norm{(\eta''\da u^{k})}_{\ld} +\frac{c}{s^{1+\al}(t)} \norm{(\eta'\poch\da u^{k})}_{\ld}\\
&\leq c\norm{\da u^{k}}_{\ld} +c \norm{A(t) u^{k}}_{\ld}\\
\end{split}
\end{equation}
Additionally, taking into account that $u^k=w^k-w^k(1,t)x^\al$, with $w^k(\cdot,t) \in {_0}H^{1+\al}(0,1)$, applying Proposition~\ref{domOpFrac} and \cite[Theorem 9.8]{Lions}, we have
\begin{equation}\label{acot0002-1}
\begin{split}
\norm{\da u^{k}}_{\ld} &=  \norm{\p^\al u^{k}}_{\ld}\leq \norm{\p^\al w^{k}}_{\ld}+|w^{k}(1,t)|\norm{\p^\al x^\al}_{\ld}\\
&\leq c\norm{w^{k}}_{{_0}H^\al(0,1)}+c\norm{w^{k}}_{C([0,1])}\leq c\norm{w^{k}}_{{_0}H^1(0,1)}+c\norm{w^{k}}_{C([0,1])}\\
&= c\norm{w^{k}}_{{_0}H^{1+\al}(0,1)}+c\norm{w^{k}}_{C([0,1])}
\leq c||u^k||_{\widetilde{\mathcal{D}}_\al}.
\end{split}
\end{equation}
Hence, by \eqref{acot0002}, \eqref{acot0002-1} and Lemma \ref{equivL2Dal} we deduce that
\begin{equation}\label{acot0002-2}
\frac{1}{s^{1+\al}(t)}\norm{\p^{\al}(\eta'\da u^{k})}_{L^{2}(0,1)}\leq c \norm{A(t) u^{k}}_{\ld}.
\end{equation}

Combining \eqref{acot0001}, \eqref{acot0002-2}, and \eqref{nowef} with $\overline{\gamma}=0$, we obtain that for every $0<\gamma < \frac{1}{2}$
\eqq{
\norm{F^{k}(\cdot,t)}_{L^{2}(0,1)} \leq c\norm{A(t) u^{k}(\cdot,t)}_{L^{2}(0,1)}\leq c(t-\sigma)^{\frac{\gamma}{(1+\al)}-1}\norm{u_{\sigma}}_{H^{\gamma}(0,1)}
}{nowej}
and hence $F^{k} \in L^{1}(\sigma,T;L^{2}(0,1))$.


Recalling that $u^{k} \in C([\sigma,T];\widetilde{\mathcal{D}}_\al)$, from definition of $F^{k}$, it follows that $F^{k}\in C((\sigma,T];L^{2}(0,1))$. In fact, for $0<\tau<t\leq T$, we have that
\begin{equation}
\begin{split}
\norm{F^{k}(\cdot,t)-F^{k}(\cdot,\tau)}_{L^2(0,1)} &\leq \al \norm{\eta}_{W^{1,\infty}(0,1)} \norm{I^{1-\al}[A(t)u^{k}(\cdot,t)-A(\tau)u^{k}(\cdot,\tau)]}_{L^2(0,1)}\\
&\quad +\norm{\p^{\al}\left(\frac{1}{s^{1+\al}(\tau)}(\eta'\da u^{k})(\cdot,\tau)- \frac{1}{s^{1+\al}(t)}(\eta'\da u^{k})(\cdot,t)\right)}_{L^2(0,1)}
\end{split}
\end{equation}
The first term tends to zero because fractional integral is bounded in $L^{2}$ and $A(\cdot)u^{k} \in C([\sigma,T];L^{2}(0,1))$. Furthermore, since $u^{k} \in C([\sigma,T];\widetilde{\mathcal{D}}_\al)$, we have $\eta'\da u^{k} \in C([\sigma,T];{}_{0}H^{1}(0,1))$. In particular, taking into account \eqref{assums}, $\frac{1}{s^{1+\al}(t)}\eta'\da u^{k} \in C([\sigma,T];L^{2}(0,1))$. Then, by Lemma \ref{equivL2Dal} and \eqref{acot0000-0}, we deduce that $F^{k} \in C((\sigma,T];L^{2}(0,1))$, and \eqref{item1} holds.

To verify \eqref{item2} we proceed as follows. Since $\vf^k\in C_0^\infty(0,1)\subseteq {_0}H^{1+2\al}(0,1)$, it follows that
$\p^{\al}\vf^k=\frac{\p}{\p x}I^{1-\al}\vf^k\in H^{1+\al}(0,1)$
and then, $\p^{\al}\vf^k\cdot \eta\in H_0^{1+\al}(0,1)\subseteq \widetilde{\mathcal{D}}_\al$.

Finally, we will prove \eqref{item3}.
From \eqref{darl} we have  $\p^{\al}u^{k} = \da u^{k}$.
Thus, it is enough to show that $\da u^{k} \cdot \eta \in C([\sigma,T];L^{2}(0,1))\cap C((\sigma,T];\widetilde{\mathcal{D}}_\al)$.
Recall that $\da u^{k} \cdot \eta \in C([\sigma,T];L^{2}(0,1))$.  Let us show that $\da u^{k} \cdot \eta \in C((\sigma,T];\widetilde{\mathcal{D}}_\al)$.

We note that for any $0 < \beta <1+\al$
\[
\norm{\poch (\eta\da u^{k})(\cdot,t)}_{H^{\beta}(0,1)}  \leq  \norm{\eta'\da u^{k}(\cdot,t)}_{H^{\beta}(0,1)} + \norm{\eta\poch \da u^{k}(\cdot,t)}_{H^{\beta}(0,1)}.
\]
Applying Theorem \ref{IneqInterpSp} and estimate (\ref{r2o}) we obtain that
\begin{equation}\label{jknsw}
\begin{split}
\norm{\eta\poch \da u^{k}(\cdot,t)}_{H^{\beta}(0,1)}&\leq\norm{\poch \da u^{k}(\cdot,t)}_{H^{\beta}\left(\frac{\varepsilon}{2},\omega^*\right)}\leq \norm{\poch \da u^{k}(\cdot,t)}_{[L^{2}(0,1),\widetilde{\mathcal{D}}_\al]_{\frac{\beta}{1+\al}}}
\leq c(t-\sigma)^{-\frac{\beta}{1+\al}}\norm{\vf^{k}}_{\widetilde{\mathcal{D}}_\al}.
\end{split}
\end{equation}
Moreover,
$$\norm{\eta' \da u^{k}(\cdot,t)}_{H^{\beta}(0,1)}\leq c\norm{ \da u^{k}(\cdot,t)}_{H^{\beta}\left(\frac{\varepsilon}{2},\omega^*\right)}.$$
Besides, if $\beta\in (0,1]$,
\begin{equation}
\begin{split}
\norm{ \da u^{k}(\cdot,t)}_{H^{\beta}\left(\frac{\varepsilon}{2},\omega^*\right)}&\leq c\norm{ \da u^{k}(\cdot,t)}_{H^{1}\left(\frac{\varepsilon}{2},\omega^*\right)}\leq c\norm{ \da u^{k}(\cdot,t)}_{H^{1}(0,1)}\\
&\leq  c\left(\norm{\da u^{k}(\cdot,t)}_{L^{2}(0,1)}+\norm{ \poch\da u^{k}(\cdot,t)}_{L^{2}(0,1)}\right).\\
\end{split}
\end{equation}
Furthermore,
\begin{equation}
\begin{split}
\norm{\da u^{k}(\cdot,t)}_{L^{2}(0,1)}&\leq \norm{\partial^\al w^{k}(\cdot,t)}_{L^{2}(0,1)}+c\norm{w^k}_{C([0,1])}\\
&\leq c\norm{w^{k}(\cdot,t)}_{{_0}H^{1+\al}(0,1)}=c\norm{u^{k}(\cdot,t)}_{\widetilde{\mathcal{D}}_\al}\leq c\norm{A(t) u^{k}(\cdot,t)}_{L^{2}(0,1)}.\\
\end{split}
\end{equation}
Hence,
\begin{equation}\label{acotDDD}
\norm{ \da u^{k}(\cdot,t)}_{H^{\beta}\left(\frac{\varepsilon}{2},\omega^*\right)}\leq c\norm{A(t) u^{k}(\cdot,t)}_{L^{2}(0,1)}, \quad 0<\beta \leq 1.
\end{equation}

Consider now $\beta \in (1,1+\al)$. Since $\beta-1\in (0,1)$, by \eqref{jknsw} and \eqref{acotDDD}, it follows that
\begin{equation}\label{acotDDD2}
\begin{split}
\norm{ \da u^{k}(\cdot,t)}_{H^{\beta}\left(\frac{\varepsilon}{2},\omega^*\right)}&= \norm{ \da u^{k}(\cdot,t)}_{L^2\left(\frac{\varepsilon}{2},\omega^*\right)}+\norm{ \poch\da u^{k}(\cdot,t)}_{H^{\beta-1}\left(\frac{\varepsilon}{2},\omega^*\right)}\\
&\leq \norm{ \da u^{k}(\cdot,t)}_{H^{\beta-1}\left(\frac{\varepsilon}{2},\omega^*\right)}+c\norm{ \eta\poch\da u^{k}(\cdot,t)}_{H^{\beta-1}(0,1)}\\
&\leq c \left(\norm{A(t)u^{k}(\cdot,t)}_{L^2(0,1)}+(t-\sigma)^{-\frac{\beta-1}{1+\al}}\norm{\vf^{k}}_{\widetilde{\mathcal{D}}_\al}\right),
\end{split}
\end{equation}

Thus, by Lemma \eqref{equivL2Dal}, \eqref{jknsw}, \eqref{acotDDD} and \eqref{acotDDD2}, we deduce
\eqq{
\norm{\poch (\eta\da u^{k})(\cdot,t)}_{H^{\beta}(0,1)} \leq c \left(\norm{A(t)u^{k}(\cdot,t)}_{L^2(0,1)}+(t-\sigma)^{-\frac{\beta}{1+\al}}\norm{\vf^{k}}_{\widetilde{\mathcal{D}}_\al}\right),
}{jkns}

which implies 
$$\poch (\eta\da u^{k}) \in L^{\infty}_{loc}((\sigma,T];H^{\beta}(0,1)) \m{ for every }0 <\beta <1+\al.$$
Taking into account the definition of $\eta$, and the fact that $||\eta D^{\alpha}u^{k}||_{\widetilde{\mathcal{D}}_\alpha} = ||\eta D^{\alpha}u^{k}||_{{_0}H^{1+\alpha}(0,1)}$ we deduce that
\[
\norm{\eta\da u^{k}(\cdot,t) - \eta\da u^{k}(\cdot,\tau)}_{\widetilde{\mathcal{D}}_\al}^{2}=\norm{\eta\da u^{k}(\cdot,t) - \eta\da u^{k}(\cdot,\tau)}_{_{0}H^{1+\al}(0,1)}^{2}
\]
\eqq{
= \norm{\eta\da u^{k}(\cdot,t) - \eta\da u^{k}(\cdot,\tau)}_{\ld}^{2} + \norm{\poch(\eta\da u^{k})(\cdot,t) - \poch(\eta\da u^{k})(\cdot,\tau)}_{_{0}H^{\al}(0,1)}^{2}.
}{000001}

Since $\da u^{k} \cdot \eta \in C([\sigma,T];L^{2}(0,1))$, it follows that the first norm in \eqref{000001} tends to zero as $\tau\rightarrow t$.

For the last term in \eqref{000001}, we use the characterization \cite[(3.17)]{RoRyVe2022} and \cite[Corollary 1.2.7.]{Lunardi} to deduce that
\[
 \norm{\poch(\eta\da u^{k})(\cdot,t) - \poch(\eta\da u^{k})(\cdot,\tau)}_{_{0}H^{\al}(0,1)}^{2}
\]
\[
  \leq c \norm{\poch(\eta\da u^{k})(\cdot,t) - \poch(\eta\da u^{k})(\cdot,\tau)}_{_{0}H^{\beta}(0,1)}^{\frac{\al}{\beta}} \norm{\poch(\eta\da u^{k})(\cdot,t) - \poch(\eta\da u^{k})(\cdot,\tau)}_{\ld}^{1-\frac{\al}{\beta}}
\]
for every $\al<\beta<1+\al$. We note that, by (\ref{jkns}), for every $\sigma< \tau,t \leq T$ the first norm is bounded  while the second tends to zero as $\tau \rightarrow t$ because $u^{k} \in C([\sigma,T];\widetilde{\mathcal{D}}_\al)$. Hence, $\da u^{k} \cdot \eta \in C((\sigma,T];\widetilde{\mathcal{D}}_\al)$.

Furthermore, $\da u^{k} \cdot \eta \in C^1((\sigma,T];\ld)$. Indeed,  $A(t)u^{k} \in C((\sigma,T];\ld)$ and by (\ref{jknsw}) $A(t)u^{k} \in L^{\infty}_{loc}((\sigma,T];H^{\beta}(0,1))$ for every $\beta \in (0,1+\al)$. Thus, applying again the interpolation estimate, in particular we obtain that $A(t)u^{k} \in C((\sigma,T];H^{1}(0,1))$ and hence $u^{k}_{t} \in C((\sigma,T];H^{1}(0,1))$ which implies $\p^{\al}u^{k}_{t}\cdot \eta = \da u^{k}_{t} \cdot \eta \in C((\sigma,T];\ld)$, and condition \eqref{item3} is verified.\\

We can apply now Proposition \ref{comi} to problem \eqref{noweh} to obtain that
$\p^{\al} u^{k} \cdot \eta$ satisfies the integral identity
\eqq{
(\p^{\al} u^{k} \cdot \eta)(x,t) = G(t,\sigma)(\p^{\al} \vf^{k} \cdot \eta)(x) + \int_{\sigma}^{t}G(t,\tau)F^{k}(x,\tau)d\tau.
}{nowei}

We fix $\gamma_0 \in (0,1+\al)$. Since $\p^{\al} u^{k} \cdot \eta \in L^2(0,1)$ and $F^{k} \in L^{1}(\sigma,T;L^{2}(0,1))$, by \eqref{r4o} and Theorem \ref{IneqInterpSp} it follows that
\begin{equation}
\begin{split}
\norm{(\p^{\al} u^{k} \cdot \eta)(\cdot,t)}_{H^{\gamma_0}(0,1)}&=\norm{(\p^{\al} u^{k} \cdot \eta)(\cdot,t)}_{H^{\gamma_0}\left(\frac{\varepsilon}{2},\omega^*\right)}\leq c\norm{(\p^{\al} u^{k} \cdot \eta)(\cdot,t)}_{[L^{2}(0,1),\widetilde{\mathcal{D}}_\al]_{\frac{\gamma_0}{1+\al}}}\\
&\leq c\norm{G(t,\sigma)(\p^{\al} \vf^{k} \cdot \eta)}_{[L^{2}(0,1),\widetilde{\mathcal{D}}_\al]_{\frac{\gamma_0}{1+\al}}} + c\int_{\sigma}^{t}\norm{G(t,\tau)F^{k}(\cdot,\tau)}_{[L^{2}(0,1),\widetilde{\mathcal{D}}_\al]_{\frac{\gamma_0}{1+\al}}} d\tau\\
&\leq c(t-\sigma)^{-\frac{\gamma_0}{1+\al}}\norm{\p^{\al} \vf^{k} \cdot \eta}_{L^{2}(0,1)} + c \int_{\sigma}^{t}(t-\tau)^{-\frac{\gamma_0}{1+\al}}\norm{F^{k}(\cdot,\tau)}_{L^{2}(0,1)}d\tau.
\end{split}
\end{equation}

Using the estimate (\ref{nowej}) we get that for every $0<\bar{\gamma}<\frac{1}{2}$ there holds
\begin{equation}
\begin{split}
\norm{(\p^{\al} u^{k} \cdot \eta)(\cdot,t)}_{H^{\gamma_0}(0,1)} &\leq c(t-\sigma)^{-\frac{\gamma_0}{1+\al}}\norm{\p^{\al} \vf^{k}}_{L^{2}(0,\omega_{*})} +c \int_{\sigma}^{t}(t-\tau)^{-\frac{\gamma_0}{1+\al}}(\tau-\sigma)^{\frac{\bar{\gamma}}{1+\al}-1}d\tau
\norm{u_{\sigma}}_{H^{\bar{\gamma}}(0,1)}\\
&= c(t-\sigma)^{-\frac{\gamma_0}{1+\al}}\norm{\p^{\al} \vf^{k}}_{L^{2}(0,\omega_{*})} +c (t-\sigma)^{\frac{\bar{\gamma}-\gamma_0}{1+\al}}B\left(\frac{\bar{\gamma}}{1+\al},1-\frac{\gamma}{1+\al}\right) \norm{u_{\sigma}}_{H^{\bar{\gamma}}(0,1)},
\end{split}
\end{equation}
where $B$ is the Beta function.

By Proposition \ref{domOpFrac} and \eqref{noweb} we have $\norm{\p^{\al} \vf^{k}}_{L^{2}(0,\omega_{*})}\leq c\norm{\vf^{k}}_{{_0}H^{\al}(0,\omega_{*})}\leq
c\norm{u_\sigma}_{{_0}H^{\al}(0,\omega_{*})}$ for $k$ big enough.

Hence,
\begin{equation}\label{acot0000001}
\begin{split}
\norm{(\p^{\al} u^{k} \cdot \eta)(\cdot,t)}_{H^{\gamma_0}(0,1)}\leq c\left[(t-\sigma)^{-\frac{\gamma_0}{1+\al}}\norm{u_\sigma}_{{_0}H^{\al}(0,\omega_{*})} + (t-\sigma)^{\frac{\bar{\gamma}-\gamma_0}{1+\al}}\norm{u_{\sigma}}_{H^{\bar{\gamma}}(0,1)}\right]
\end{split}
\end{equation}

Applying \eqref{darl} and by \cite[Proposition 2.1]{Hitchhiker} we deduce that
\begin{equation}\label{acot00005}
\norm{(D^{\al} u^{k} \cdot \eta)(\cdot,t)}_{H^{\gamma_0}(0,1)}\leq c (t-\sigma)^{-\frac{\gamma_0}{1+\al}}\norm{u_\sigma}_{{_0}H^{\al}(0,1)}.
\end{equation}
Since $\eta$ is smooth, we may apply \eqref{acot0000001} with $\eta = \eta'$ to obtain 
\begin{equation}\label{acot00006}
\norm{(D^{\al} u^{k}\cdot \eta')(\cdot,t)}_{H^{\gamma_0}(0,1)}\leq c(t-\sigma)^{-\frac{\gamma_0}{1+\al}}\norm{u_{\sigma}}_{{_0}H^{\al}(0,1)}.
\end{equation}

Then, combining \eqref{acot00005} and \eqref{acot00006}, we deduce that for all $0<\gamma<\al$
\begin{equation}
\begin{split}
\norm{\left(\poch D^{\al} u^{k} \cdot \eta\right)(\cdot,t)}_{H^{\gamma}(0,1)}&= \norm{\poch(D^{\al} u^{k} \cdot \eta)(\cdot,t)}_{H^{\gamma}(0,1)}+\norm{(D^{\al} u^{k} \cdot \eta')(\cdot,t)}_{H^{\gamma}(0,1)}\\
&\leq c (t-\sigma)^{-\frac{1+\gamma}{1+\al}}\norm{u_\sigma}_{{_0}H^{\al}(0,1)}.
\end{split}
\end{equation}

Hence, by \eqref{assums},
\begin{equation}\label{acotAtukloc}
\norm{A(t) u^{k} (\cdot,t)}_{H^{\gamma}\left(\frac{\varepsilon}{2},\omega^*\right)}\leq c(t-\sigma)^{-\frac{1+\gamma}{1+\al}}\norm{u_\sigma}_{H^\al(0,1)}.
\end{equation}


Finally, from the weak lower semicontinuity of the norm, we deduce that
\begin{equation}
\norm{A(t) u(\cdot,t)}_{H^{\gamma}\left(\frac{\varepsilon}{2},\omega^*\right)}\leq c\norm{A(t) u^k(\cdot,t)}_{H^{\gamma}\left(\frac{\varepsilon}{2},\omega^*\right)}\leq c(t-\sigma)^{-\frac{1+\gamma}{1+\al}}\norm{u_\sigma}_{{_0}H^\al(0,1)}
\end{equation}
where $c=c(\al,b,M,\varepsilon,\omega,T,\gamma)$. Since $\varepsilon$ and $\omega$ are arbitrary numbers, \eqref{regloc1} holds.

\end{proof}

Now we are able to increase the regularity of solutions to (\ref{StefanProblemwithsCil}) in the interior.

\begin{lemma}\label{third}
 Let $\al\in \left(\frac{1}{2},1\right)$, $v_{0} \in \widetilde{\mathcal{D}}_\al$, and  $v$ be the mild solution to (\ref{StefanProblemwithsCil}) obtained in Theorem \ref{ExistUnicMildSol}. Then the following results hold:

\begin{enumerate}


\item There exists $\beta>\frac{1}{2}$ such that
\begin{equation}\label{regvbeta}
v \in L^{\infty}_{loc}(0,T;H^{\beta+1+\al}_{loc}(0,1)) \mbox{ and } \p^{\al}v_{x} \in L^{\infty}_{loc}(0,T;H^{\beta}_{loc}(0,1)).
\end{equation}

\item Moreover, there exists $\beta>\frac{1}{2}$ such that
\begin{equation}\label{regv-cont}
v \in C((0,T];H^{\beta+1+\al}_{loc}(0,1)) \mbox{ and } \p^{\al}v_{x} \in C((0,T];H^{\beta}_{loc}(0,1)).
\end{equation}

\end{enumerate}

\end{lemma}

\begin{proof}

1.  Let $\beta \in \left(\frac{1}{2},\al\right)$.
By Definition \ref{mildsolstand}
\eqq{
v(x,t) = G(t,0)v_{0}(x) + \int_0^t G(t,\sigma)f(x,\sigma)\dd \sigma.
}{vcf}

In order to use Theorem \ref{IneqInterpSp}, we introduce the parameters $0<\varepsilon<\omega<1$.
We apply the operator $A(t)$ to (\ref{vcf}) and estimate its norm
\begin{equation}
\norm{A(t)v}_{H^{\beta}(\varepsilon,\omega)} \leq \norm{A(t)G(t,0)v_{0}}_{H^{\beta}(\varepsilon,\omega)} + \int_0^t \norm{A(t)G(t,\sigma)f(\cdot,\sigma)}_{H^{\beta}(\varepsilon,\omega)}\dd \sigma.
\end{equation}

Let $\sigma\in [0,T]$ and denote $w(x,t):=G(t,\sigma)f(x,\sigma)$. Then,
$$w_t(x,t)=\frac{\p}{\p t}G(t,\sigma)f(x,\sigma)=A(t)G(t,\sigma)f(x,\sigma)=A(t)w(x,t) \quad \text{ and }$$
$$w(x,\sigma)=G(\sigma,\sigma)f(x,\sigma)=f(x,\sigma).$$
Hence, by Theorem \ref{IneqInterpSp}, Lemma \ref{lemusinewdomain} and \eqref{r2o}, it follows that
\begin{equation}\label{acotAGv0}
\begin{split}
\norm{A(t)v}_{H^{\beta}(\varepsilon,\omega)} &\leq \norm{A(t)G(t,0)v_{0}}_{H^{\beta}(\varepsilon,\omega)} + \int_0^t \norm{A(t)G(t,\sigma)f(\cdot,\sigma)}_{H^{\beta}(\varepsilon,\omega)}\dd \sigma\\
&\leq c\norm{A(t)G(t,0)v_{0}}_{[L^2(0,1),\widetilde{\mathcal{D}}_\al]_\frac{\beta}{1+\al}} + c\int_0^t (t-\sigma)^{-\frac{1+\beta}{1+\al}}\norm{f(\cdot,\sigma)}_{{_0}H^\al(0,1)}\dd \sigma\\
&\leq ct^{-\frac{\beta}{1+\al}}\norm{v_{0}}_{\widetilde{\mathcal{D}}_\al}  + c\frac{1+\al}{\al - \beta} t^{\frac{\al - \beta}{1+\al}}\norm{f}_{L^\infty(0,T;{_0}H^{\al}(0,1))}.
\end{split}
\end{equation}

Hence, by \eqref{assums} it follows that $\p^{\al}v_{x} \in L^{\infty}_{loc}(0,T;H^{\beta}_{loc}(0,1))$ for all $\frac{1}{2}<\beta<\al$.

On the other hand, since $v(\cdot,t) \in \widetilde{\mathcal{D}}_{\alpha}$ for every $t$,
$$v(x,t)=w(x,t)-w(1,t)x^\al, \quad w(\cdot,t)\in {_0}H^{1+\al}(0,1),$$
and then, $w_x(\cdot,t)\in {_0}H^{\al}(0,1)$. Additionally, since $\poch \da x^\al =0$, it follows that $A(t)v=A(t)w$, thus $\p^{\al}w_{x}(\cdot,t)=\p^{\al}v_{x}(\cdot,t)\in H^{\beta}_{loc}(0,1)$ for all $\frac{1}{2}<\beta<\al$.
Hence, applying Lemma \ref{local}, we deduce that $w_x \in L^{\infty}_{loc}(0,T;H^{\beta+\al}_{loc}(0,1))$.
Finally
\begin{equation}\label{acotvxHba}
\begin{split}
\norm{v_x(\cdot,t)}_{H^{\beta+\al}(\varepsilon,\omega)} &\leq \norm{w_x(\cdot,t)}_{H^{\beta+\al}(\varepsilon,\omega)}+\norm{\al w(1,t)x^{\al-1}}_{H^{\beta+\al}(\varepsilon,\omega)}\\
\end{split}
\end{equation}
and taking into account that $x^{\al-1}\in H^{\beta+\al}(\varepsilon,\omega)$, we deduce that $v_x \in L^{\infty}_{loc}(0,T;H^{\beta+\al}_{loc}(0,1))$ and \eqref{regvbeta} holds.

2. Theorem \ref{ExistUnicMildSol} states that $v \in  C([0,T];\widetilde{\mathcal{D}}_\al)$. Since for arbitrary $0<\varepsilon<\omega<1$ and for every $0<\overline{\beta} < \beta$ there holds \\
\[
H^{\overline{\beta}+\al+1}(\varepsilon,\omega) = [H^{1+\al}(\varepsilon,\omega),H^{\beta+\al+1}(\varepsilon,\omega)]_{\frac{\overline{\beta}}{\beta}},
\]
we may estimate by the interpolation theorem (\cite[Corollary 1.2.7]{Lunardi})
\[
\norm{v(\cdot,t)-v(\cdot,\tau)}_{H^{\overline{\beta}+\al+1}(\varepsilon,\omega)}
\hspace{-0.1cm} \leq \hspace{-0.1cm} c\norm{v(\cdot,t)-v(\cdot,\tau)}_{H^{1+\al}(\varepsilon,\omega)}^{1-\frac{\overline{\beta}}{\beta}} \hspace{-0.1cm}
\norm{v(\cdot,t)-v(\cdot,\tau)}_{H^{\beta+\al+1}(\varepsilon,\omega)}^{\frac{\overline{\beta}}{\beta}},
\]
where $c =  c(\beta,\overline{\beta},\varepsilon)$.
By Lemma \ref{third}-1, there exists $\beta >\frac{1}{2}$ such that the second norm on the right hand side above is bounded on every compact interval contained in $(0,T]$, and the first term tends to zero because $v\in C([0,T];\widetilde{\mathcal{D}}_\al)$.
We conclude that $v \in C((0,T];H^{\bar{\beta}+1+\al}_{loc}(0,1))$.

In order to obtain the claim for $\p^{\al}v_{x}$ we note that, applying again the interpolation theorem we obtain that for every $0<\varepsilon<\omega<1$, $0<\tau < t \leq T$ and every $0<\overline{\beta} < \beta$,
\begin{equation}
\begin{split}
\norm{\p^{\al}v_{x}(\cdot,t) - \p^{\al}v_{x}(\cdot,\tau)}_{H^{\overline{\beta}}(\varepsilon,\omega)}\leq c\norm{\p^{\al}v_{x}(\cdot,t) - \p^{\al}v_{x}(\cdot,\tau)}_{L^{2}(0,1)}^{1-\frac{\overline{\beta}}{\beta}}\norm{\p^{\al}v_{x}(\cdot,t) - \p^{\al}v_{x}(\cdot,\tau)}_{H^{\beta}(\varepsilon,\omega)}^{\frac{\overline{\beta}}{\beta}}.
\end{split}
\end{equation}
Then, the first norm on the right hand side tends to zero as $\tau \rightarrow t$, while the second one is bounded on every compact interval contained in $(0,T]$ due to Lemma \ref{third}-1.
Then, we may choose $\bar{\beta} > \frac{1}{2}$ such that $\p^{\al}v_{x} \in C((0,T];H^{\bar{\beta}}_{loc}(0,1))$.
\end{proof}

\begin{coro}\label{regost}
 Let us assume that $v_0 \in \widetilde{\mathcal{D}}_\al$.
 Then,  the solution to \eqref{StefanProblemwithsCil} obtained in Theorem \ref{ExistUnicMildSol} satisfies that
  \eqq{
  v \in C([0,1]\times [0,T]).
  }{vciaglosc}
\end{coro}

\begin{proof}
It is a direct consequence of the embedding $\widetilde{\mathcal{D}}_\al \hookrightarrow C([0,1])$.
\end{proof}


\begin{coro}\label{w2b}
Under the assumptions of Lemma \ref{third} for every $\al \in \left(\frac{1}{2},1\right)$ there exists $\beta > \frac{1}{2}$ such that for every $0<\varepsilon<\omega<1$ there holds $$v_{x}\in C((0,T];H^{\al+\beta}(\varepsilon,\omega)).$$
\end{coro}

\begin{remark}
The regularity in Corollary \ref{w2b} is a fundamental tool for the next section, where the maximum principle for the fractional diffusion equation \eqref{StefanProblem0}-$(i)$ is used to claim the existence and uniqueness of the solution to the fractional Stefan problem \eqref{StefanProblem0}. 
\end{remark}


\subsection{Existence of a classical solution to the moving boundary problem}

We are able to prove existence and regularity of solutions to the moving boundary problem \eqref{StefanProblemwithsCil}.

\begin{theo}\label{ExistsUnique}
Let $b>0$ and $\al\in \left(\frac{1}{2},1\right)$. Suppose that $s$ satisfies \eqref{assums}, and $v_0\in \widetilde{\mathcal{D}}_\al$, where $v_0(x)=u_0(bx)$.
Then, there exists a unique solution $u$ to \eqref{MBP} such that $u \in C(\overline{Q_{s,T}})$, $u_t, \frac{\partial}{\partial x}D^\al u\in C(Q_{s,T})$, and for all $t\in (0,T]$, $u(\cdot,t)\in AC[0,s(t)]$, $u_t(\cdot,t), \frac{\partial}{\partial x}D^\al u(\cdot,t)\in L^2(0,s(t))$. Furthermore, there exists $\beta>\frac{1}{2}$ such that for all $t\in (0,T]$ and $0<\varepsilon<\omega<s(t)$ we have $u_x(\cdot,t)\in H^{\al+\beta}(\varepsilon,\omega)$.
\end{theo}

\begin{proof}


Let $v$ be the mild solution to \eqref{StefanProblemwithsCil} given by Theorem \ref{ExistUnicMildSol}.
By hypothesis $v_0\in \widetilde{\mathcal{D}}_\al$. Then, by Theorem \ref{ExistUnicMildSol}, there exists a unique solution $v\in C([0,T];\widetilde{\mathcal{D}}_\al)$ to \eqref{mildSw}. Since $v(\cdot,t)\in \widetilde{\mathcal{D}}_\al$ for all $t\in [0,T]$, $v$ satisfies \eqref{StefanProblemwithsCil}-($ii$).  Moreover, by Lemma \ref{regmildsol}, $v$ verifies \eqref{StefanProblemwithsCil}-($i$). Hence, $v$ is a solution to \eqref{StefanProblemwithsCil}.

On the other hand, $v$ has the regularity given by Lemma \ref{third} and Corolary \ref{regost}. In particular, $v\in C([0,1]\times [0,T])$.
Moreover, by Corolary \ref{w2b} there exists $\beta>\frac{1}{2}$ such that, for $0<\varepsilon<\omega<1$, $v_x\in C((0,T];H^{\al+\beta}(\varepsilon,\omega))$.\\

Define $u(x,t)=v\left(\frac{x}{s(t)},t\right)$.  Then, $u$ is a solution to \eqref{MBP}.
Taking into account that $v\in C([0,1]\times[0,T])$ and \eqref{assums}, it follows that $u\in C(\overline{Q_{s,T}})$. On the other hand, since $v_x\in C((0,T];H^{\al+\beta}(\varepsilon,\omega))$ for every $0<\varepsilon<\omega<1$, we have $u_x(\cdot,t)\in H^{\al+\beta}(\varepsilon,\omega)$ for all $t\in (0,T]$ and for every $0<\varepsilon<\omega<s(t)$.
Also, we know that $v\in C([0,T],\widetilde{\mathcal{D}}_\al)$. Then, we deduce that $u(\cdot,t)\in AC[0,s(t)]$ for all $t\in[0,t]$.

We note that, by Proposition \ref{domOpFrac} we have that $\frac{\partial}{\partial x}D^\al v(\cdot,t)=\partial^\al v_x(\cdot,t)\in L^2(0,1)$, for all $t\in (0,T]$. Then, if we define $p=\frac{x}{s(t)}$,
$$\frac{\partial}{\partial x}D^\al u(x,t)=\frac{1}{s^{1+\al}(t)} \left(\frac{\partial}{\partial p}D^\al v\right)(p,t),$$
and by \eqref{assums} we deduce that $u_t(\cdot,t),\frac{\partial}{\partial x}D^\al u(\cdot,t)\in L^2(0,s(t))$, for all $t\in(0,T]$.

In addition, by Lemma \ref{third}-4 we have that $\p^{\al}v_{x} \in C((0,T];H^{\beta}_{loc}(0,1))$, from where, by the Sobolev embedding \cite[Chapter 9, Theorem 9.8]{Lions} it follows that $\p^{\al}v_{x} \in C((0,1)\times(0,T])$. Hence, $\p^{\al}u_{x} \in C(Q_{s,T})$ and we conclude that $\frac{\p}{\p x} D^\al u, u_t\in C(Q_{s,T})$.

We prove now the uniqueness of solution. Let $u_1, u_2$ be solutions to \eqref{MBP} with the regularity given in the theorem. We define $u=u_1-u_2$. Then, $u$ is a solution to \eqref{MBP} for $u_0\equiv 0$. Thus, multiplying by $u$, integrating over $Q_{s,t_0}$ for $0<t_0<T$ the equation \eqref{MBP}-(i), and applying Fubini's Theorem, we obtain

\begin{equation}\label{acot1}
\begin{split}
\int_0^{t_0} \int_0^{s(\tau)}u_t(x,\tau)\cdot u(x,\tau)\dd x\dd \tau - \int_0^{t_0} \int_0^{s(\tau)}\frac{\partial}{\partial x}D^\al u(x,\tau)\cdot u(x,\tau)\dd x\dd \tau &=0\\
\frac{1}{2}\int_0^{s(t_0)}\int_{s^{-1}(x)}^{t_0}\frac{d}{dt}(u(x,\tau)^2)\dd \tau\dd x + \int_0^{t_0} \int_0^{s(\tau)}-\frac{\partial}{\partial x}D^\al u(x,\tau)\cdot u(x,\tau)\dd x\dd \tau &=0\\
\frac{1}{2}\int_0^{s(t_0)}|u(x,t_0)|^2\dd x + \int_0^{t_0} \left(-\frac{\partial}{\partial x}D^\al u(\cdot,t),u(\cdot,t)\right)\dd \tau &=0.
\end{split}
\end{equation}

From (\ref{koerd}) we infer that
\begin{equation}\label{desigpi2}
\left(-\frac{\partial}{\partial x}D^\al u(\cdot,t),u(\cdot,t)\right)\geq  c_\al ||u(\cdot,t)||^2_{H^{\frac{1+\al}{2}}(0,s(t))},
\end{equation}
where $c_\al$ is a positive constants that depend on $\al$. Then,

\begin{equation}\label{acot2}
\begin{split}
0&=\frac{1}{2}\int_0^{s(t_0)}|u(x,t_0)|^2\dd x + \int_0^{t_0} \left(-\frac{\partial}{\partial x}D^\al u(\cdot,t),u(\cdot,t)\right)\dd \tau
\geq 0
\end{split}
\end{equation}
from where $\displaystyle\frac{1}{2}\int_0^{s(t_0)}|u(x,t_0)|^2\dd x=0$. Hence, since $u\in C(\overline{Q_{s,T}})$, it follows that $u(\cdot,t_0)=0$. Being $0<t_0<T$ an arbitrary number, it results that $u\equiv 0$, and then $u_1\equiv u_2$.

\end{proof}


\section{Existence of a classical solution to the Stefan Problem with Dirichlet boundary condition}

Let us return to the problem
\begin{equation}\label{StefanProblem}
\begin{array}{lll}
(i) & u_t-\frac{\partial}{\partial x}D^\al u=0, & \text{ in } Q_{s,T}\\
(ii) & u(0,t)=0, u(t,s(t))=0, &   t\in (0,T),\\
(iii) & u(x,0)=u_0(x), &  0<x<s(0)=b,\\
(iv) & \dot{s}(t)=-\lim\limits_{x\rightarrow s(t)^-}(D^\al u)(x,t), & t\in (0,T).
\end{array}
\end{equation}

\begin{remark}\label{regDxal}
Note that, from \cite[Proposition 2.9]{RoRyVe2022}, if $u$ is the solution to \eqref{StefanProblem} given by Theorem \ref{ExistsUnique} for a fixed function $s$, we deduce that $D^\al u(\cdot,t) \in C(0,s(t)]$, for all $t\in (0,T]$, and then $\lim\limits_{x\rightarrow s(t)^-}(D^\al u)(x,t)$ is well defined for all $t\in (0,T)$.
\end{remark}


\begin{theo}\label{ppioextremo}
Let $u$ be a solution to the equation
$$ {u}_t-\frac{\p}{\p x} D_x^{\al} u=f\quad \text{in }\, Q_{s,T} $$
such that $u\in C(\overline{Q_{s,T}})$, $ u_t,\frac{\partial}{\partial x} \, D_x^{\al}u\in C(Q_{s,T})$ and  $u(\cdot,t)\in AC[0,s(t)]$ $\forall t\in (0,t)$. Also suppose that $s$ verifies \eqref{assums} and $u$ verifies the local regularity condition
\eqq{
\exists \, \beta > \frac{1}{2}\, \text{ such that, for every } t\in (0,T) \mbox{ and every } 0<\varepsilon<\omega < s(t), \hd  u_{x}(\cdot,t) \in H^{\al+\beta}(\varepsilon,\omega).
}{regumax}

Then $u$ attains its maximum on $\partial\Gamma_{s,T}$ if $f\leq 0$, and $u$ attains its minimum on $\partial\Gamma_{s,T}$ if $f\geq 0$.

\end{theo}

%

%

The following results are analogous to \cite[Lemma 4.14]{PhDKasia} 
and \cite[Lemma 4.16]{PhDKasia} but we recall them for the benefit of the reader.

\begin{lemma}\label{Dalfaleqzero}
Let $u$ be the solution to \eqref{MBP} given by Theorem \ref{ExistsUnique}, for a fixed function $s$. Then $(D^\al u)(s(t),t)\leq 0$ for all $t\in(0,T]$. Furthermore, if $u_0\not\equiv 0$, then for all $t\in(0,T]$, $(D^\al u)(s(t),t)< 0$.
\end{lemma}

\begin{lemma}\label{lemaaaa}
Let $u$ be the solution to \eqref{MBP} given by Theorem \ref{ExistsUnique}.
Suppose that  $s$ satisfies \eqref{assums}, and $u_0$ verifies
\begin{equation}\label{desigu0}
0\leq u_0(x)\leq \frac{M}{2\Gamma(1+\al)} (b^\al-x^\al), \quad x\in [0,b].
\end{equation}
where $M$ is the constant in \eqref{assums}.
Then,
\begin{equation}\label{desigresult1}
(D^\al u)(s(t),t)\geq -\frac{M}{2}, \quad t\in (0,T),
\end{equation}
and
\begin{equation}\label{desigresult2}
0\leq u(x,t)\leq \frac{M}{2\Gamma(1+\al)} (s^\al(t)-x^\al), \quad (x,t)\in Q_{s,T}.
\end{equation}
\end{lemma}

%
%

We establish now an integral condition equivalent  to \eqref{StefanProblem}$-(iv)$.
\begin{prop}\label{condicionintegralDgeneral}
Let $(u,s)$ be a solution to \eqref{StefanProblem} having the regularity in Theorem \ref{ppioextremo}.
Then, the condition \eqref{StefanProblem}$-(iv)$ is equivalent to the following integral condition
\begin{equation}\label{CondInt2general}
s^2(t)=b^2+2\int_0^b x u_0(x)\dd x-2\int_0^t I^{1-\al} u(s(\tau),\tau)\dd \tau-2\int_0^{s(t)}x u(x,t)\dd x.
\end{equation}
\end{prop}

\begin{proof}
Suppose that  $u$ satisfies \eqref{MBP} and let us define 
$h(x)=\begin{cases}0 & 0<x\leq b, \\ s^{-1}(x) & x>b.\end{cases}$.
Then, multiplying equation \eqref{MBP}$-(i)$ by $x$ and integrating in time  from $h(x)$ 
to $t$ we get 
$$xu(x,t)-xu(x,h(x))=\int_{h(x)}^t x\frac{\p}{\p x}D^\al u(x,\tau)\dd\tau $$
Now we integrate in space, apply Fubini's theorem and integrate by parts to obtain
$$ \int_0^{s(t)}xu(x,t)\dd x - \int_0^{s(t)}xu(x,h(x))\dd x=\int_0^{s(t)}\int_{h(x)}^t x\frac{\p}{\p x}D^\al u(x,\tau)\dd\tau\dd x  $$
$$\int_0^{s(t)}xu(x,t)\dd x - \int_0^{b}xu(x,0)\dd x=\int_0^{t} s(\tau)D^\al u(s(\tau),\tau)  \dd\tau-\int_0^t\int_0^{s(t)}D^\al u(x,\tau) \dd x \dd \tau $$
Finally, from (\ref{equiv2-1}) we deduce 
\begin{equation}\label{equiv-IntCond} \int_0^{s(t)}xu(x,t)\dd x - \int_0^{b}xu(x,0)\dd x=\int_0^{t} s(\tau)D^\al u(s(\tau),\tau)  \dd\tau-\int_0^t I^{1-\al}u(s(\tau),\tau)\dd \tau\end{equation}
Now, if $(u,s)$ is  a solution to the Stefan problem \eqref{StefanProblem} then \eqref{CondInt2general} easily follows by placing \eqref{StefanProblem}$-(iv)$ in \eqref{equiv-IntCond}.
 For the converse, suppose that the pair $(u,s)$ verifies $\eqref{MBP}$ and $\eqref{CondInt2general}$
Then by differentiating both sides in \eqref{CondInt2general} and using \eqref{MBP}$-(ii)$  we get 

\begin{equation}
\begin{split}
2s(t)\dot{s}(t)&=-2I^{1-\al}u(s(t),t) -2\int_0^{s(t)}x u_t(x,t)\dd x=-2I^{1-\al}u(s(t),t)-2\int_0^{s(t)}x \frac{\partial}{\partial x}D^\al u(x,t)\dd x
\end{split}
\end{equation}
and the claim follows by integration by parts. 

\end{proof}

Now we present the monotone dependence upon data result. It is analogous to \cite[Theorem 3]{RoRyVe2022}, but we need to modify the estimates in the proof, due to the new integral condition given in Proposition \ref{condicionintegralDgeneral}. In the proof we will use the following generalized Gronwall's Lemma.

\begin{lemma}\label{GronwallGen2}\cite[Theorem 2.1]{Kyung}
Let $p>q\geq 0$, $x(t)$ and $h(t)$ be real-valued non-negative continuous functions defined on $I = [0, \infty)$, and $n(t)$ be a positive non-decreasing continuous function defined on $I$. If the following inequality holds 
$$x^p(t) \leq n^p(t) + \int_0^t h(s)x^q(s)ds, \quad \forall t\in I,$$
then
$$x(t)\leq n(t)\left[1+\frac{p-q}{p}\int_0^t h(s)n^{-(p-q)}(s)ds\right], \quad \forall t\in I.$$
\end{lemma}

\begin{theo}\label{dependence}
Let $s_i$, $i=1,2$ verify \eqref{assums}, and $u^i$ be the solutions to \eqref{MBP} given by Theorem~\ref{ExistsUnique}, corresponding to $s_i$, $b_i$ and $u_0^i$ for $i=1,2$ respectively. If $b_1\leq b_2$ and $u_0^1 \leq u_0^2$, then for every $t\in [0,T]$ we have $s_1(t)\leq s_2(t)$.
\end{theo}

\begin{proof}

We consider two cases.

{\bf Case 1:} $b_1<b_2$, $u_0^1\leq u_0^2$ and $u_0^1\not\equiv u_0^2$ on $[0,b_1]$.

It is identical to the one in \cite[Theorem 3]{RoRyVe2022}.
%
%
%
%
%

\vspace{0.3cm}

{\bf Case 2:} $b_1\leq b_2$, $u_0^1\leq u_0^2$.

Let $\delta>0$, and $u_0^\delta$ be a smooth funtion on $[0,b_2+\delta]$ such that
$$u_0^\delta\equiv 0, \text{ on } \left[b_2+\frac{\delta}{2},b_2+\delta\right], \quad u_0^\delta\geq u_0^2, \text{ on } [0,b_2],$$
$$\max_{x\in [0,b_2]}(u_0^\delta(x)-u_0^2(x))=\delta, \quad \max_{x\in \left[b_2,b_2+\frac{\delta}{2}\right]}u_0^\delta(x)=\delta.$$
We denote by $(u^\delta, s_\delta)$ a solution to \eqref{StefanProblem} corresponding to the data $b_2+\delta$ and $u_0^\delta$.

By the preceding case, we have that $s_1\leq s_\delta$ and $s_2\leq s_\delta$. Moreover, by Proposition \ref{condicionintegralDgeneral} we deduce that
%
\begin{equation}
\begin{split}
s_\delta^2 (t)-s_2^2(t)&=2\delta b_2 + \delta^2+2\int_{b_2}^{b_2+\frac{\delta}{2}} x u_0^\delta (x)\dd x+2\int_0^{b_2} x (u_0^\delta - u_0^2)(x)\dd x\\
&\quad +2\int_0^t [I^{1-\al} u^2(s_2(\tau),\tau)-I^{1-\al} u^\delta(s_\delta(\tau),\tau)]\dd \tau\\
&\quad -2\int_0^{s_2(t)} x (u^\delta -u^2)(x,t)\dd x -2 \int_{s_2(t)}^{s_\delta(t)} x u^\delta(x,t)\dd x 
\end{split}
\end{equation}
On the other hand, by Theorem \ref{ppioextremo}, we have $u^\delta \geq 0$, $u^\delta-u^2\geq 0$. Additionally, $||u_0^\delta-u_0^2||_{L^\infty(0,b_2)}=\delta$. Therefore, 
\begin{equation}\label{acotdiffcuads}
\begin{split}
s_\delta^2 (t)-s_2^2(t)&\leq 2\delta b_2 + \delta^2+\delta\left(b_2\delta+\frac{\delta^2}{4}\right)+b_2^2\delta+2\int_0^t [I^{1-\al} u^2(s_2(\tau),\tau)-I^{1-\al} u^\delta(s_\delta(\tau),\tau)]\dd \tau.\\
\end{split}
\end{equation}

We note that,
\begin{equation}
\begin{split}
B(\tau)&:= I^{1-\al} u^2(s_2(\tau),\tau)-I^{1-\al} u^\delta(s_\delta(\tau),\tau)\\
&\leq\frac{1}{\Gamma(1-\al)}\left[\int_0^{s_2(\tau)} \frac{u^2(x,\tau)}{(s_2(\tau)-x)^\al}-\frac{u^\delta(x,\tau)}{(s_2(\tau)-x)^\al}\dd x\right.\\
&\quad \left.+\int_0^{s_2(\tau)} \frac{u^\delta(x,\tau)}{(s_2(\tau)-x)^\al}-\frac{u^\delta(x,\tau)}{(s_\delta(\tau)-x)^\al}\dd x \right],\\
\end{split}
\end{equation}
and by \eqref{desigresult2} and Remark \ref{acotS},
$$0\leq u^\delta(x,\tau)\leq \frac{M}{2\Gamma(1+\al)}(s_\delta^\al(\tau)-x^\al)\leq \frac{M(b_2+\delta+MT)^\al}{2\Gamma(1+\al)}.$$
Moreover, we know that $s_2\leq s_\delta$ and $u^2\leq u^\delta$, thus
\begin{equation}\label{acotB}
\begin{split}
B(\tau)
&\leq  \frac{M(b_2+\delta+MT)^\al}{2\Gamma(1+\al)\Gamma(2-\al)}(s_\delta(\tau)-s_2(\tau))^{1-\al}
\end{split}
\end{equation}

Additionally,
\begin{equation}\label{acotdiffs}
2b_2(s_\delta (t)-s_2(t))\leq(s_\delta (t)+s_2(t))(s_\delta (t)-s_2(t))= s_\delta^2 (t)-s_2^2(t).
\end{equation}

Combining \eqref{acotdiffcuads}, \eqref{acotB} and \eqref{acotdiffs}, we obtain 
\begin{equation}\label{acotparaGronwall}
s_\delta (t)-s_2(t)\leq \delta\left(1+\frac{b_2}{2}+\frac{\delta}{2b_2}+\frac{\delta}{2}+\frac{\delta^2}{8b_2}\right) +\frac{M(b_2+\delta+MT)^\al}{b_2\Gamma(1+\al)\Gamma(2-\al)}\int_0^t (s_\delta(\tau)-s_2(\tau))^{1-\al}\dd \tau.
\end{equation}

We want to apply a generalized Gr\"onwall's inequality to \eqref{acotparaGronwall}. More precisely, we apply Theorem \ref{GronwallGen2} with $x(t)=s_\delta (t)-s_2(t)$, $n(t)=\delta\left(1+\frac{b_2}{2}+\frac{\delta}{2b_2}+\frac{\delta}{2}+\frac{\delta^2}{8b_2}\right)$, $h(t)=\frac{M(b_2+\delta+MT)^\al}{b_2\Gamma(1+\al)\Gamma(2-\al)}$, $p=1$ and $q=1-\al$. Then, 

\begin{equation}
\begin{split}
s_\delta (t)-s_2(t)&\leq \delta\left(1+\frac{b_2}{2}+\frac{\delta}{2b_2}+\frac{\delta}{2}+\frac{\delta^2}{8b_2}\right)+\al\delta^{1-\al}\left(1+\frac{b_2}{2}+\frac{\delta}{2b_2}+\frac{\delta}{2}+\frac{\delta^2}{8b_2}\right)^{1-\al}\frac{M(b_2+\delta+MT)^\al}{b_2\Gamma(1+\al)\Gamma(2-\al)} T,
\end{split}
\end{equation}
which leads to

\begin{equation}\label{Gronwall3}
s_1(t)\leq s_\delta (t)\leq s_2(t)+\delta\left(1+\frac{b_2}{2}+\frac{\delta}{2b_2}+\frac{\delta}{2}+\frac{\delta^2}{8b_2}\right)+\al\delta^{1-\al}\left(1+\frac{b_2}{2}+\frac{\delta}{2b_2}+\frac{\delta}{2}+\frac{\delta^2}{8b_2}\right)^{1-\al}\frac{M(b_2+\delta+MT)^\al}{b_2\Gamma(1+\al)\Gamma(2-\al)} T,
\end{equation}

Finally, taking $\delta\rightarrow 0$, it follows that
$$s_1 (t)\leq s_2(t), \quad \forall t\in [0,T].$$

\end{proof}

By Theorem \ref{ExistsUnique} and Theorem \ref{dependence} the uniqueness result follows easily.

\begin{coro}\label{unicidad}
The solution to \eqref{StefanProblem} is unique in a class of functions satisfying the regularity assumptions from Theorem \ref{ExistsUnique}.
\end{coro}


We are ready to prove existence and uniqueness of solution to \eqref{StefanProblem}. The proof is similar to the one given in \cite[Theorem 3.14]{RoRyVe2022}, where an operator $P$ was defined in such a way that its fixed point satisfies an equivalent integral condition for a Stefan Problem similar to \eqref{StefanProblem} where a fractional Neumann condition was considered.
However, we cannot use the same operator $P$, since \eqref{CondInt2general} is not equivalent to the integral condition obtained in \cite[Theorem 2.13]{RoRyVe2022}.

\begin{theo}\label{existsanduniqueSP}
Let $b,T>0$, $\al\in (1/2,1)$. Furthermore, let $u_0$ be a nonnegative function defined on $[0,b]$ such that $v_0(x):=u_0(bx)$ verifies $v_0\in \widetilde{\mathcal{D}}_\al$. Moreover, suppose that there exists a positive constant $M$ such that,
\begin{equation}\label{acotu0}
u_0(x)\leq \frac{M}{2\Gamma(1+\al)} (b^\al-x^\al), \quad x\in[0,b].
\end{equation}

Then, there exists a unique solution $(u,s)$ to \eqref{StefanProblem},
such that $s\in C^{0,1}([0,T])$, $0<\dot{s}(t)\leq M$ for all $t\in (0,T]$, $u \in C(\overline{Q_{s,T}})$, $u_t, \frac{\partial}{\partial x}D^\al u\in C(Q_{s,T})$, and for all $t\in (0,T]$, $u(\cdot,t)\in AC[0,s(t)]$, $u_t(\cdot,t), \frac{\partial}{\partial x}D^\al u(\cdot,t)\in L^2(0,s(t))$. Furthermore, there exists $\beta>\frac{1}{2}$ such that for all $t\in (0,T]$ and $0<\varepsilon<\omega<s(t)$ we have $u_x(\cdot,t)\in H^{\al+\beta}(\varepsilon,\omega)$.
\end{theo}

\begin{proof}

We follow the ideas introduced in \cite{Andreucci}, that was used also in \cite{RoRyVe2022} and \cite{Rys:2020}, but proposing a new operator $P$, destined to be related to the the new integral condition given in Proposition \ref{condicionintegralDgeneral}.

We define the set
$$\Sigma:=\{s\in C^{0,1}[0,T]: 0\leq\dot{s}\leq M, s(0)=b\}.$$
 We will apply the Schauder fixed point Theorem. We note that $\Sigma$ is a compact and convex subset of the Banach space $C([0,T])$ with the maximum norm.  

For every $s\in \Sigma$ we define 
\begin{equation}\label{Ps}
(Ps)(t)=\left(b^2 -2\int_0^t D^{\al} u(s(\tau),\tau)s(\tau)\dd \tau\right)^\frac{1}{2},
\end{equation}
where $u$ is the unique solution to the MBP \eqref{MBP} which exists and is unique due to  Theorem \ref{ExistsUnique}.

Let's prove that $P:\Sigma\to\Sigma$. 
Define $$ H(t)= b^2 -2\int_0^t D^{\al} u(s(\tau),\tau)s(\tau)\dd \tau.$$
Note that $H(0)=b^2$ and  from Lemma  \ref{Dalfaleqzero}
 $H'(t)\geq 0$ for every $t>0$. Then  $Ps$ is well defined and from the regularity of $u$ we deduce that $Ps$ is a continuous function. 
 
From the other side,  
from Lemma \ref{lemaaaa} we have 
\begin{equation}\label{step}
0\leq (Ps)'(t)=\frac{-2D^\al u(s(t),t)s(t)}{2H^{1/2}(t)}\leq \frac{M}{2b}(b+MT)\leq M \quad \text{ if } T\leq \frac{b}{M} \end{equation}
and then $Ps \in \Sigma$.


Now, in order to prove that the operator P is continuous, we use the relation (\ref{equiv-IntCond}) to state that $P$ may be written as follows 
\begin{equation}\label{new-Ps} (Ps)(t)=\left( b^2+2j(t)\right)^{1/2} \end{equation}
where
$$j(t)=\int_0^b x u_0 (x)\dd x -\int_0^t I^{1-\al} u(s(\tau),\tau)\dd \tau-\int_0^{s(t)} x u(x,t)\dd x > 0.$$

Let $s_1,s_2\in \Sigma$ and $s_{min}(t):=\min\{s_1(t),s_2(t)\}$, $s_{max}(t):=\max\{s_1(t),s_2(t)\}$, $i:[0,T]\rightarrow \{1,2\}$ with $i(t)=1$ if $s_{max}(t)=s_1(t)$ and $i(t)=2$ otherwise.

Let $u_1,u_2$ be solutions to \eqref{MBP} given by Theorem \ref{ExistsUnique}, corresponding to $s_1$ and  $s_2$ respectively. We define $v(x,t)=u_1(x,t)-u_2(x,t)$.
Then, $v$ satisfies
\begin{equation}
\begin{array}{lll}
(i) & v_t-\frac{\partial}{\partial x}D^\al v=0, & \text{ in } Q_{s_{min},T}\\
(ii) & v(0,t)=0, &   t\in (0,T),\\
(iii) & v(s_{min}(t),t)=(-1)^{i(t)+1}u_{i(t)}(s_{min}(t),t) &   t\in (0,T),\\
(iv) & v(x,0)=0, &  0<x<b.
\end{array}
\end{equation}
By Theorem \ref{ppioextremo}, $v$ attains its maximum and minimum on the parabolic boundary. 
Hence,
$$\max_{Q_{s_{min},T}} |v| \leq \max_{t\in[0,T]}|u_{i(t)}(s_{min}(t),t)|,$$
and applying \eqref{desigresult2}, we obtain that
\begin{equation}\label{acotABSv}
\max_{Q_{s_{min},T}} |v|\leq \frac{M}{2\Gamma(1+\al)} \max_{t\in[0,T]}|s^\al_1(t)-s^\al_2(t)|.
\end{equation}

Using the real estimate
\begin{equation}\label{realestim}
b^\al -a^\al\leq \frac{\al}{a^{1-\al}}(b-a), \text{ for } 0<a<b,\end{equation}
it follows that

\begin{equation}\label{acotdiffPs1bis}
\begin{split}
|(P s_2)(t)-(P s_1)(t)| &= \left| \left( b^2+2j_2(t)\right)^{1/2}-\left( b^2+2j_1(t)\right)^{1/2}\right|\leq \frac{2 \left| j_2(t)-j_1(t)\right|}{2b}
 \\ 
 &\leq\frac{1}{b}\left(\int_0^t |I^{1-\al} u_1(s_1(\tau),\tau)-I^{1-\al}u_2(s_2(\tau),\tau)|\dd \tau+\int_0^{s_{min}(t)}\hspace{-0.7cm}x|v(x,t)|\dd x+\int_{s_{min}(t)}^{s_{max}(t)}\hspace{-1cm} x u_{i(t)}(x,t) \dd x\right)
 \end{split}
 \end{equation}
Now, 
\begin{equation}\label{F1F2F3}
\begin{split}
&\Gamma(1-\al)|I^{1-\al} u_{2}(s_{2}(\tau),\tau)-I^{1-\al}u_1(s_{1}(\tau),\tau)|\\
&\leq\int_0^{s_{min}(\tau)}\frac{M}{2\Gamma(1+\al)}(s^\al_{max}(\tau)-x^\al)\left(\frac{1}{(s_{min}(\tau)-x)^\al}-\frac{1}{(s_{max}(\tau)-x)^\al}\right)\dd x\\
&\quad +\frac{M}{2\Gamma(1+\al)}\int_{s_{min}(\tau)}^{s_{max}(\tau)}\frac{s^\al_{max}(\tau)-x^\al}{(s_{max}(\tau)-x)^\al}\dd x+\int_0^{s_{min}(\tau)} \frac{|v(x,\tau)|}{(s_{min}(\tau)-x)^\al}\dd x\\
&\leq\frac{M}{2\Gamma(1+\al)}\int_0^{s_{min}(\tau)}(s_{max}^\al(\tau)-x^\al)\left(\frac{1}{(s_{min}(\tau)-x)^\al}-\frac{1}{(s_{max}(\tau)-x)^\al}\right)\dd x\\
&\quad+\frac{M}{2\Gamma(1+\al)}\int_{s_{min}(\tau)}^{s_{max}(\tau)}\frac{s_{max}^\al(\tau)-x^\al}{(s_{max}(\tau)-x)^\al}\dd x+\frac{M}{2\Gamma(1+\al)} \max_{\tau\in[0,T]}|s^\al_1(\tau)-s^\al_2(\tau)|\int_0^{s_{min}(\tau)}(s_{min}(\tau)-x)^{-\al}\dd x\\
&=:F_1+F_2+F_3.
\end{split}
\end{equation}
Applying the sustitution $w=\frac{x}{s_{min}(\tau)}$ in $F_1$, and using again \eqref{realestim},
it follows that
\begin{equation}
\begin{split}
F_1&= \frac{M}{2\Gamma(1+\al)}\int_0^1 (s_{max}(\tau)^\al-w^\al s_{min}(\tau)^\al)\left(\frac{1}{(s_{min}(\tau)-w s_{min}(\tau))^\al}-\frac{1}{(s_{max}(\tau)-w s_{min}(\tau))^\al}\right)s_{min}(\tau)\dd w\\
&= \frac{M}{2\Gamma(1+\al)}\int_0^1 \left(\left(\frac{s_{max}(\tau)}{s_{min}(\tau)}\right)^\al-w^\al \right)\left((1-w)^{-\al}-\left(\frac{s_{max}(\tau)}{s_{min}(\tau)}-w\right)^{-\al}\right)s_{min}(\tau)\dd w\\
&\leq \frac{M}{2\Gamma(1+\al)}\int_0^1 \al w^{\al-1}\left(\frac{s_{max}(\tau)}{s_{min}(\tau)}-w\right)\left((1-w)^{-\al}-\left(\frac{s_{max}(\tau)}{s_{min}(\tau)}-w\right)^{-\al}\right)s_{min}(\tau)\dd w\\
&= \frac{M}{2\Gamma(1+\al)}\int_0^1 \al w^{\al-1}(1-w)^{-\al}\left(\left(\frac{s_{max}(\tau)}{s_{min}(\tau)}-w\right)-(1-w)^{\al}\left(\frac{s_{max}(\tau)}{s_{min}(\tau)}-w\right)^{1-\al}\right)s_{min}(\tau)\dd w\\
&\leq \frac{M}{2\Gamma(1+\al)}\int_0^1 \al w^{\al-1}(1-w)^{-\al}\left(\left(\frac{s_{max}(\tau)}{s_{min}(\tau)}-w\right)-(1-w)^{\al}(1-w)^{1-\al}\right)s_{min}(\tau)\dd w\\
&= \frac{\al M}{2\Gamma(1+\al)}(s_{max}(\tau)-s_{min}(\tau))\int_0^1  w^{\al-1} (1-w)^{-\al}\dd w\\
&\leq \frac{M \Gamma(1-\al)}{2} \max_{t\in[0,T]}|s_1(t)-s_2(t)|.
\end{split}
\end{equation}

On the other hand, by \eqref{assums} and \eqref{realestim}, it is easy to see that $F_2$ and $F_3$ in \eqref{F1F2F3} are bounded by $\max_{t\in[0,T]}|s_1(t)-s_2(t)|$ as follows,
$$F_2\leq \frac{ M(b+MT)^{1-\al}}{b^{1-\al}\Gamma(\al)}\max_{t\in[0,T]}|s_{max}(t)-s_{min}(t)|,$$
$$F_3\leq \frac{ M(b+MT)^{1-\al}}{b^{1-\al}(1-\al)\Gamma(\al)}\max_{t\in[0,T]}|s_{max}(t)-s_{min}(t)|.$$

Hence, if  $c=\max\left\{\frac{M \Gamma(1-\al)}{2}, \frac{M(b+MT)^{1-\al}}{b^{1-\al}(1-\al)\Gamma(\al)} \right\}$, we have
\begin{equation}\label{F1F2F3-2}
\Gamma(1-\al)|I^{1-\al} u_{2}(s_{2}(\tau),\tau)-I^{1-\al}u_1(s_{1}(\tau),\tau)|\leq c \max_{t\in[0,T]}|s_{max}(t)-s_{min}(t)|,
\end{equation}

Applying \eqref{desigresult2}, \eqref{acotABSv} and \eqref{F1F2F3-2} to \eqref{acotdiffPs1bis}, we obtain that
\begin{equation}
\begin{split}
&|(P s_2)(t)-(P s_1)(t)|\\
&\leq \frac{1}{b}\left(\frac{cT}{\Gamma(1-\al)}\max_{t\in[0,T]}|s_{max}(t)-s_{min}(t)|\right.\\
&\quad \left. +\frac{\al M b^{\al-1}}{2\Gamma(1+\al)}\max_{t\in[0,T]}|s_{max}(t)-s_{min}(t)|\int_0^{s_{min}(t)}x \dd x + \frac{M}{2\Gamma(1+\al)}\int_{s_{min}(t)}^{s_{max}(t)} x (s_{max}^\al(t)-x^\al)\dd x\right)\\
&\leq \frac{1}{b}\left(\frac{cT}{\Gamma(1-\al)}\max_{t\in[0,T]}|s_{max}(t)-s_{min}(t)|\right.\\
&\quad \left. +\frac{\al M b^{\al-1}(b+MT)^2}{4\Gamma(1+\al)}\max_{t\in[0,T]}|s_{max}(t)-s_{min}(t)|+ \frac{M s_{max}^\al(t)}{2\Gamma(1+\al)}\frac{s_{max}(t)^2-s_{min}(t)^2}{2}\right)\\
&\leq \frac{1}{b}\left(\frac{cT}{\Gamma(1-\al)} +\frac{\al M b^{\al-1}(b+MT)^{2}}{4\Gamma(1+\al)}+ \frac{M s_{max}^\al(t)}{2\Gamma(1+\al)}\frac{s_{max}(t)+s_{min}(t)}{2}\right)\max_{t\in[0,T]}|s_{max}(t)-s_{min}(t)|\\
\end{split}
\end{equation}
We conclude that, for $C=\frac{1}{b}\left(\frac{cT}{\Gamma(1-\al)}+\frac{\al Mb^{\al-1}(b+MT)^{2}}{4\Gamma(1+\al)}+\frac{M(b+MT)^{1+\al}}{2\Gamma(1+\al)}\right)$
$$|(P s_1)(t)-(P s_2)(t)|\leq C  \max_{t\in[0,T]}|s_1(t)-s_2(t)|.$$

Hence, $P$ is continuous in the maximum norm. From the Schauder fixed point Theorem, we obtain that for $T \leq \frac{b}{M}$, there exists a unique $s\in \Sigma$ such that $Ps=s$ and the pair $(u,s)$ verifies \eqref{MBP} and \eqref{CondInt2general}. Applying Proposition \ref{condicionintegralDgeneral} we deduce that $(u,s)$ is a solution to  \eqref{StefanProblem} and the uniqueness follows from Corollary \ref{unicidad}.
Thus, we proved Theorem (\ref{existsanduniqueSP}) for $T\leq\frac{b}{M}$.

We can extend the result for larger T by a standard argument, by considering a time step of length $\frac{b}{2M} $ and taking as the initial condition value of solution obtained in a previous step 
in point $t_i=\frac{b}{4M}$. Due to Theorem \ref{ExistsUnique}, in each step the initial condition is regular enough to repeat the proof of the first step.
In fact, let $(u^0,s^0)$ be the solution to \eqref{StefanProblem} for $T = \frac{b}{2M}$, and consider a new problem given by \eqref{StefanProblem} with $u^1_0(x)= u^0(x,\frac{b}{4M})$, $s^1(0)=b^1:=s^0\left(\frac{b}{4M}\right)\geq b$. Then $u^1_0$ satisfy the assumption \eqref{acotu0} due to Lemma~\ref{lemaaaa}, and we can repeat the previous reasoning obtaining a solution $(u^1,s^1)$ to \eqref{StefanProblem} in $[0,\frac{b}{2M}]$, but we have started from $t_1=\frac{b}{4M}$. Note that from Proposition \ref{condicionintegralDgeneral} $s^0$ and $s^1$ verify both the fractional Stefan condition \eqref{StefanProblem}$-(iv)$ and from Corollary \ref{unicidad}  we have that the solutions  $(u^0,s^0)$ and $(u^1,s^1)$ agree in the region $\left\{ 0\leq x\leq s^0(t)=s^1(t-\frac{b}{4M}), \frac{b}{4M}\leq t<\frac{b}{2M} \right\}$.  
Then, we can state that the pair $(u,s)$ given by $s(t)=\begin{cases} s^0(t), & 0\leq t <\frac{b}{4M}\\ s^1(t-\frac{b}{4M}) & \frac{b}{4M}\leq t<\frac{b}{2M}   \end{cases}$ and $u(x,t)=\begin{cases}u^0(x,t) & 0\leq x\leq s(t), 0\leq t<\frac{b}{4M}\\ u^1(x,t-\frac{b}{4M}) & 0\leq x\leq s(t), \frac{b}{4M} \leq t<\frac{b}{2M}  \end{cases}$ is a solution to \eqref{StefanProblem} in $Q_{s,\frac{b}{2M}}$ and has the regularity requested in the claim. 

Finally, by repeating this argument $k$ times for $k$ such that $k\frac{b}{4M}>T$, we can obtain a solution $(u,s)$ to 
\eqref{StefanProblem}, defining $s:[0,T]\mapsto \bbR$ and $u:Q_{s,T}\mapsto \bbR$ as follows
$$s(t)=s^i\left(t-i\frac{b}{4M}\right), \text{ for } t\in \left(i\frac{b}{4M},(i+1)\frac{b}{4M}\right],\quad i=0,...,k$$
$$u(x,t)=u^i\left(x,t-i\frac{b}{4M}\right), \text{ for } x\in[0,s^i(t)],\quad t\in \left(i\frac{b}{4M},(i+1)\frac{b}{4M}\right],\quad i=0,...,k.$$

\end{proof}

\section{Conclusions}

We have studied a space-fractional Stefan problem with the Dirichlet homogeneus boundary conditions, and have proven the existence and uniqueness of a regular solution. In the proof, we utilized several established facts for similar fractional Stefan problems, and developed additional results, including an estimate between fractional Sobolev spaces and interpolation spaces, and an integral equation equivalent to the Stefan condition. We applied the latter to prove the monotone dependence on data. Aditionally, we adapted regularity results given for a similar problem with a fractional Neumann condition.

It remains an open problem to study the existence and uniqueness of the solution to problem \eqref{StefanProblem0} with $b=0$.

\section{Acknowledgements}
The first and third author were partly supported by the Project INV-006-00030 from Universidad
Austral and PICT-2021-I-INVI-00317.  The second author was partly supported
by the National Sciences Center, Poland, through grant 2017/26/M/ST1/00700.


\begin{thebibliography}{99}


\bibitem{AcoBorBruMaas}
Acosta G., Borthagaray J.P., Bruno O.P., Maas M.:
{Regularity theory and high order numerical methods for the (1D)-fractional Laplacian}.
Math. Comput., Vol. 87, pp:1821--1857, (2016).

\bibitem{Andreucci}
Andreucci D.: {Lecture notes on the Stefan problem}. Lecture notes, Universitá da
Roma La Sapienza, Italy, (2004).

\bibitem{BaeuLuksMeer}
Baeumer B., Luks T., Meerschaert MM.:
{Space-time fractional Dirichlet problems}.
Mathematische Nachrichten, Vol. 291, pp:2516--2535, (2018).

\bibitem{BaFiRos}
Barrios B., Figalli A., Ros-Oton X.:
{Global regularity for the free boundary in the obstacle problem for the fractional Laplacian}.
American Journal of Mathematics, Vol.140 (2), pp:415--447, (2018).

\bibitem{Brezis:2011}
Br\'ezis, H.:
{Functional Aalisys, Sobolev Spaces and Partial Differential Equations}.
Springer, (2011).

\bibitem{CaSoVa}
Caffarelli L.A., Soria F., Vázquez J.L.:
{Regularity of solutions of the fractional porous medium flow}.
Journal of the European Mathematical Society, Vol. 15 (5), pp:1701--1746, (2013).


\bibitem{CaBeYo}
Capart H., Bellal M., Young D.L.:
{Self-similar evolution of semi-infinite alluvial channels with moving boundaries}.
Journal of Sedimentary Research, Vol. 77, pp:13--22, (2007).



\bibitem{TeEnVa2020}
del Teso F., Endal J., Vázquez J.L.:
{The one-phase fractional Stefan problem}.
Mathematical Models and Methods in Applied Sciences, Vol. 31 (1), pp:83--131, (2021).

\bibitem{Hitchhiker}
Di Nezza, E., Palatucci, G., Valdinoci, E.:
{Hitchhiker's guide to the fractional Sobolev spaces}.
Bulletin des Sciences Mathématiques, (2012).

\bibitem{Kyung}
El-Owaidy, H., Ragab, A., Eldeeb, A.:
{On Some New Nonlinear Integral Inequalities of Gronwall-Bellman Type}.
Kyungpook Mathematical Journal, (2015).

\bibitem{GLY}
Gorenflo R., Luchko Y., Yamamoto M.:
{Time-fractional diffusion equation in the fractional {S}obolev spaces}.
Fractional Calculus and Applied Analysis, Vol. 28, No. 3, 799-820, (2015).

\bibitem{GeVoMiDa:2013}
Gruber C. A., Vogl C. J., Miksis M. J., Davis S. H.:
{Anomalous diffusion models in the presence of a moving interface}.
Interfaces and {F}ree {B}oundaries, Vol. 15, pp:181--202, (2013).

\bibitem{JuTry}
Juric D., Tryggvason, G.:
{A front tracking method for dendritic solidification}.
Journal of Computational Physics, Vol. 123, pp: 127--148, (1996).

\bibitem{KaRa}
Karma, A., Rappel, W. J.:
{Quantitative phase-field modeling of dendritic growth in two and three dimensions}.
Physical Review E, Vol. 57, pp:4323--4349, (1998).



\bibitem{KuYa}
Kubica, A., Yamamoto, M.:
{Initial-boundary value problems for fractional diffusion equations with time-dependent coefficients}.
Fractional Calculus and Applied Analysis, Vol. V21(2), pp:276--311, (2018).



\bibitem{Lions}
Lions, J. L., Magenes, E.:
{Non-Homogeneous Boundary Value Problems and Applications - Vol. 1}.
Springer-Verlag Berlin Heidelberg, (1972).


\bibitem{Lunardi}
Lunardi, A.:
{Analytic semigroups and optimal regularity in parabolic problems}.
Springer Science and Business Media , (2012).


\bibitem{MaSwePaVo}
Marr, J.G., Swenson J.B., Paola C., Voller V.R.:
{A two-diffusion model of fluvial stratigraphy in closed depositional basins}.
Basin Research, Vol. 12, pp:381--398, (2000).

\bibitem{MetKla}
Metzler R., Klafter J.:
{The random walk’s guide to anomalous diffusion: a fractional dynamics approach}.
Physics Reports, Vol. 339 pp:1--77, (2000).


\bibitem{NaRy2020}
Namba T., Rybka P.:
{On viscosity solutions of space-fractional diffusion equations of caputo type}.
SIAM Journal on Mathematical Analysis, Vol. 52 (1), pp:653--681, (2020).

\bibitem{NaRyVo2021}
Namba T., Rybka P., Voller V.R.:
{Some comments on using fractional derivative operators in
modeling non-local diffusion processes}.
Journal of Computational and Applied Mathematics, Vol. 381, pp:113040, (2021).



\bibitem{PaKu}
Pawlak K, Kubica A.:
{Characterization Of The Range Of The Fractional Integral Operator In $L^2$ Space}.
IChełmiński K, Górka P, Kubica A, editors. 20 years of the Faculty of Mathematics and Information Science. A collection of research papers in mathematical analysis and in partial differential equations, pp:141--189, (2022).

 \bibitem{Pazy}
 Pazy, A.:
 {Semigroups of Linear Operators and Applications to Partial Differential Equations}.
 Springer, (1983).


\bibitem{PoSe:2022}
Popov, A. Yu., Sedletskii, M.:
{Distribution of roots of mittag-leffler functions}.
Journal of Mathematical Sciences, Vol. 190 (2), pp:209-409, (2013).





\bibitem{RoRyVe2022}
Roscani, S., Ryszewska, K., Venturato, L.:
{A One-Phase Space-Fractional Stefan Problem With No Liquid Initial Domain}.
SIAM Journal on Mathematical Analysis, Vol. 54 (5), pp:5489--5523, (2022).

\bibitem{RoTaVe2020}
Roscani, S., Tarzia,  D., Venturato, L.:
{The similarity method and explicit solutions for the fractional space one-phase Stefan problems}.
Fractional Calculus and Applied Analysis, Vol. 8 (4), pp:495--508, (2022).

\bibitem{Rudin}
Rudin, W.:
Functional Analysis. Second Edition, International Editions, Mc-Graw-Hill, New York, (1991).

\bibitem{PhDKasia}
Ryszewska, K.:
{A semigroup approach to space-fractional diffusion and the analysis of fractional Stefan models}.
PhD. Thesis, Warsaw University of Technology, (2020).


\bibitem{Rys:2020}
Ryszewska, K.:
{A space-fractional Stefan problem}.
Nonlinear Analysis, Vol. 199, pp:112027, (2020).

\bibitem{Rys:2020-2}
Ryszewska, K.:
{An analytic semigroup generated by a fractional differential operator}.
Journal of Mathematical Analysis and Applications, 483(2), (2020).

\bibitem{Samko}
Samko, S. G., Kilbas, A. A., Marichev, O. I.:
{Fractional Integrals and Derivatives--Theory and Applications}.
Gordon and Breach, New York, (1993).



\bibitem{SchuMeBa}
Schumer R., Meerschaert M.M., Baeumer B.:
{Fractional advection-dispersion equations for modeling transport at the Earth surface}.
Journal of Geophysical Research, Vol. 114, (2009).

\bibitem{Voller2014}
Voller, V. R.:
{Fractional Stefan problems}.
International Journal of Heat and Mass Transfer, Vol. 74, pp:269--277, (2014).

\bibitem{Voller2011}
Voller, V. R.:
{On a fractional derivative form of the Green–Ampt infiltration model}.
Advances in Water Resources, Vol. 34, pp:257--262, (2011).

\bibitem{Yagi}
Yagi, A.:
{Abstract parabolic evolution equations and their applications}.
Springer, (2009).

\bibitem{Yama:2022}
Yamamoto, M.:
{Fractional Calculus and Time-Fractional Differential Equations: Revisit and Construction of a Theory}.
Mathematics, Vol. 10 (5), 698, (2022).

\bibitem{maxreg}
Zacher R.:
{Maximal regularity of type Lp for abstract parabolic volterra equations}.
Journal of Evolution Equations, Vol. 5, pp:79--103, (2005).

\end{thebibliography}
\end{document}